\documentclass[10pt]{amsart}

\RequirePackage{standalone}
\RequirePackage[
  a4paper,
  margin=1.2in,
  headheight=12pt,
  headsep=5mm,
  footskip=9mm,
  marginparwidth=24mm,
  marginparsep=4mm
]{geometry}
\RequirePackage[utf8]{inputenc}
\RequirePackage[T1]{fontenc}
\RequirePackage{lmodern}
\RequirePackage{microtype}
\RequirePackage[USenglish]{babel}
\RequirePackage{comment}
\RequirePackage[dvipsnames]{xcolor}
\RequirePackage{amsmath,amssymb,amsfonts,amsthm,mathtools}
\RequirePackage{tikz-cd}
\RequirePackage{mathrsfs}
\RequirePackage[mathcal]{eucal}
\RequirePackage[obeyFinal,colorinlistoftodos,textsize=footnotesize]{todonotes}

\definecolor{winered}{rgb}{0.8,0,0}
\definecolor{deepblue}{rgb}{0,0,0.8}
\definecolor{tocgreen}{RGB}{0,120,70}

\RequirePackage{hyperref}
\hypersetup{
  colorlinks=false,
  pdfborder={0 0 1},
  linkbordercolor=tocgreen,
  citebordercolor=tocgreen,
  urlbordercolor=tocgreen
}
\RequirePackage[capitalize,nameinlink]{cleveref}
\crefname{construction}{Construction}{Construction}

\makeatletter
\let\dstyle@contentsline\contentsline
\renewcommand*\let\dstyle@tableofcontents\tableofcontents
\renewcommand\tableofcontents{
  \hypersetup{pdfborder={0 0 0}}
  \dstyle@tableofcontents
  \hypersetup{pdfborder={0 0 1}}
}
\makeatother

\newtheoremstyle{d_plain}{6pt}{6pt}{\itshape}{}{\bfseries}{.}{0.5em}{}
\newtheoremstyle{d_def}{6pt}{6pt}{}{}{\bfseries}{.}{0.5em}{}
\newtheoremstyle{d_remark}{6pt}{6pt}{}{}{\itshape}{.}{0.5em}{}

\theoremstyle{d_plain}
\newtheorem{theorem}{Theorem}[section]
\newtheorem{lemma}[theorem]{Lemma}
\newtheorem{proposition}[theorem]{Proposition}
\newtheorem{corollary}[theorem]{Corollary}

\theoremstyle{d_def}
\newtheorem{definition}[theorem]{Definition}

\newtheorem{remark}[theorem]{Remark}

\theoremstyle{d_remark}
\newtheorem{notation}[theorem]{Notation}

\makeatletter
\providecommand{\@openbib@code}{}
\newcommand{\dstyle@numericbibitem}[2][]{\dstyle@savedbibitem{#2}}
\renewenvironment{thebibliography}[1]{
  \section*{\refname}
  \small
  \list{\@biblabel{\@arabic\c@enumiv}}{
    \settowidth\labelwidth{\@biblabel{999}}
    \leftmargin\labelwidth
    \advance\leftmargin\labelsep
    \@openbib@code
    \usecounter{enumiv}
    \let\p@enumiv\@empty
    \renewcommand\theenumiv{\@arabic\c@enumiv}
  }
  \sloppy
  \clubpenalty4000
  \@clubpenalty\clubpenalty
  \widowpenalty4000
  \sfcode`\.\@m
  \let\dstyle@savedbibitem\bibitem
  \let\bibitem\dstyle@numericbibitem
}{
  \def\@noitemerr{\@latex@warning{Empty `thebibliography' environment}}
  \endlist
}
\makeatother

\DeclareRobustCommand{\SkipTocEntry}[5]{}

\providecommand{\coheq}{\mathrel{\overset{\mathrm{\square}}{=}}}

\DeclareMathAlphabet{\mathmybb}{U}{bbold}{m}{n}
\newcommand{\1}{\mathmybb{1}}

\newcommand{\bE}{{\mathbb E}}

\newcommand{\mA}{{\mathcal A}}
\newcommand{\mB}{{\mathcal B}}
\newcommand{\mC}{{\mathcal C}}
\newcommand{\mD}{{\mathcal D}}
\newcommand{\mE}{{\mathcal E}}

\newcommand{\mI}{{\mathcal I}}
\newcommand{\mJ}{{\mathcal J}}
\newcommand{\mK}{{\mathcal K}}
\newcommand{\mL}{{\mathcal L}}
\newcommand{\mM}{{\mathcal M}}
\newcommand{\mN}{{\mathcal N}}
\newcommand{\mO}{{\mathcal O}}
\newcommand{\mP}{{\mathcal P}}

\newcommand{\mS}{{\mathcal S}}

\newcommand{\mV}{{\mathcal V}}
\newcommand{\mW}{{\mathcal W}}

\newcommand{\cst}{\operatorname{cst}}

\newcommand{\N}{{\mathrm N}}

\newcommand{\bD}{{\mathbb D}}

\newcommand{\Cat}{\mathsf{Cat}}

\newcommand{\colim}{\mathrm{colim}}

\newcommand{\Fun}{\operatorname{Fun}}

\newcommand{\h}{\mathrm{h}}
\newcommand{\hh}{\mathrm{h}_2}
\newcommand{\Ho}{\operatorname{Ho}}

\newcommand{\Hom}{\mathrm{Hom}}

\newcommand{\id}{\mathrm{id}}

\newcommand{\Grpd}{\mathsf{Grpd}}

\newcommand{\Lan}{\mathrm{Lan}}

\newcommand{\Map}{\operatorname{Map}}

\newcommand{\Mod}{{\mathrm{Mod}}}

\newcommand{\op}{\mathrm{op}}

\newcommand{\Psh}{\operatorname{Psh}}

\newcommand{\Qc}{\text{q}\Cat}

\newcommand{\res}{{\mathrm{res}}}

\newcommand{\Set}{\mathsf{Set}}

\newcommand{\sSet}{\mathsf{sSet}}

\newcommand{\Pro}{\operatorname{Pro}}

\newcommand{\tu}{{\1}}

\newcommand{\y}{\mathsf{y}}

\newcommand{\CAT}{\mathsf{CAT}}

\let\dstyledotlessi\i
\let\dstyledotlessj\j

\renewcommand{\i}{\ifmmode\tilde{i}\else\dstyledotlessi\fi}
\renewcommand{\j}{\ifmmode\tilde{j}\else\dstyledotlessj\fi}

\usepackage{stmaryrd}

\begin{document}

\raggedbottom

\title{ formal weakly enriched category theory}

\author{Giuseppe Leoncini}

\address{Department of Mathematics and Statistics,
Faculty of Science, Masaryk University,
Kotl\'a\v{r}sk\'a 267/2,
611 37 Brno, Czech Republic}

\address{Department of Mathematics ``Federigo Enriques'',
University of Milan,
Via Saldini 50,
20133 Milan, Italy}

\email{giuseppe.leoncini1@unimi.it}

\begin{abstract} 
A formal category theory is constructed (in the form of a proarrow equipment), encoding weak coherent enrichment over a monoidal model category $\mV$. We describe how basic categorical concepts formulated via the equipment translate back to enriched categories. We characterize Dwyer-Kan equivalences of enriched categories as $2$-categorical equivalences.

Specializing to either the Kan-Quillen model structure on simplicial sets, or the Quillen-Serre model structure on topological spaces, we prove that the resulting formal category theory is equivalent to the one associated with the $\infty$-cosmos of quasicategories, thereby extending the formal approach to $(\infty,1)$-categories in the sense of Riehl-Verity to encompass both simplicial and topological categories. 

A notion of classifying object, formulated internally to the equipment of $\mV$-categories, leads to enriched versions of Quillen's Theorem A. 

\end{abstract}

\maketitle
\markboth{}{}

\tableofcontents

\section{Introduction}

We study weak (also known as homotopy-coherent) enrichment through the lens of formal category theory. The process, in particular, supplies a dictionary between the formal weak theory and strict models. 

Our construction does not assume a theory of $(\infty,1)$-categories, but recovers it as a particular case (in a sense which we will make precise), in agreement with the quasicategorical presentation.

\subsubsection{A glimpse at the theme}
Following Lack's classification in the introduction of \cite{Lac10}, bicategories are one of three types of two-dimensional categories, the others being what Lack calls $2$-categories and $\Cat$-categories. Of course, the latter two are exactly the same thing, but different names reflect different ways of thinking about them; while $\Cat$-categories are really strict, $2$-categories are nominally strict but thought of as weak. From the point of view adopted in this paper, $\Cat$-categories are enriched in the trivial homotopy theory of categories (weak equivalences being isomorphisms), while $2$-categories are enriched in the canonical (or folk) homotopy theory of categories, where weak equivalences are equivalences of categories.

Quoting from Lack \cite{Lac10}, \begin{quote} $2$-category theory uses enriched category theory, but not in the simple minded way of $\Cat$-category theory; and it cuts through some of the technical nightmares of bicategories.
\end{quote} The above sentence could be read for any base of enrichment $\mV$ carrying some non-trivial homotopy theory, and we do so in this paper. There is affinity with Lack and Rosicky's work in \cite{LR16}.

A difficulty that arises when trying to model weakly enriched category theory from strict one has to do with cofibrancy; approaches leveraging the model structure on $\mV\Cat$ are complicated by the fact that being cofibrant there is a strong condition.

Our working hypotheses do not require $\mV$-categories $\mC$ to be cofibrant; we do not even need to assume that the model structure on $\mV\Cat$ exists in the first place. We do require of $\mC$ that the hom-objects $\mC(x,y)$ are both fibrant and cofibrant in the base $\mV$ for every $x,y \in \mC$; this much weaker condition seems natural to impose. We do not require that all objects of $\mV$ be cofibrant; $\mV$ is neither assumed locally presentable nor cofibrantly generated.

\begin{center}
\renewcommand{\arraystretch}{1.3} \setlength{\tabcolsep}{0pt}
\begin{tabular}{
  @{}p{0.28\linewidth}
  @{\hspace{0.03\linewidth}}p{0.24\linewidth}
  @{\hspace{0.03\linewidth}}p{0.42\linewidth}@{}
}
\hline
\textbf{\small Base of enrichment}
& \textbf{\small Model structure on it}
& \textbf{\small Formal category theory} \\[3pt]
\hline

$\mathcal V$
& Trivial
& Strictly $\mathcal V$-enriched
\\

$\mathsf{Grpd}$
& Canonical (folk)
& Weak $(2,1)$
\\

$\mathsf{Cat}$
& Canonical (folk)
& Weak $2$
\\

$\mathsf{Gray}$
& Lack \tiny \cite{Lac02} \normalsize
& Weak $3$
\\

$2\mathsf{Cat}_{\mathrm{flex}}$
& Campbell \tiny \cite{Cam26} \normalsize
& Weak $3$
\\

$\mathsf{sSet}^{\Theta_{n-1}^{\mathrm{op}}}$
& Rezk $\Theta_{n-1}\mathsf{Sp}_{k}$
  \tiny \cite{Rez10} \normalsize
& $(n+k,n),$ $\ n\geq1,k\geq0$
\\

$\mathsf{sSet}$
& Kan-Quillen
& Standard $(\infty,1)$
\\

$\mathsf{Top}_{\mathrm{cgwh}}$
& Quillen-Serre
& Standard $(\infty,1)$
\\

$\mathsf{Top}_{\mathrm{cgwh}}$
& Hurewicz-Str{\o}m \tiny \cite{Stro72} \normalsize
& Topological $(\infty,1)$
\\

$\mathsf{sSet}$
& Joyal
& $(\infty,2)$
\\

$\mathsf{sSet}^{\Theta_{n-1}^{\mathrm{op}}}$
& Rezk $\Theta_{n-1}\mathsf{Sp}_{\infty}$
  \tiny \cite{Rez10} \normalsize
& $(\infty,n)$, $n\geq1$
\\

$\mathsf{sSet}^{\Theta_\omega^{\mathrm{op}}}$
& Rezk $\Theta_{\omega}\mathsf{Sp}_{\infty}$
& $(\infty,\infty)$ (inductive)
\\

$\mathsf{sSet}^{t\Delta^{\mathrm{op}}}$ & Ozornova-Rovelli \tiny \cite{OR20} \normalsize & $(\infty,\infty)$ (inductive) \\

$\mathsf{sSet}^{\mathrm{Sm}(S)^{\mathrm{op}}}$
& Morel-Voevodsky
& Motivic
\\

$\mathsf{sSet}^{G}$
& Genuine \tiny \cite{DK84} \normalsize
& Equivariant
\\

$\mathsf{Ch}(R)$
& Hurewicz \tiny \cite{BMR14} \normalsize
& Weak dg up to chain-h-eq
\\

$\mathsf{Ch}(R)$
& Projective
& Weak dg up to quasi-iso
\\

$\Gamma\mS$
& Schwede \tiny \cite{Sch99} \normalsize
& Connective spectral \\

$\mathsf{Sp}^{\Sigma}$
& Stable \tiny \cite{HSS00} \normalsize
& Spectral
\\

$\mathsf{Mod}_E$
& Stable \tiny \cite{SS00} \normalsize
& $E$-linear spectral
\\

\hline
\end{tabular}
\end{center}

All of the examples above satisfy \ref{sec3assumption}. The monoidal structures are cartesian unless specified otherwise. $\mathsf{Gray}$ denotes strict $2$-categories with Gray's tensor product, while $2\mathsf{Cat}_{\mathrm{flex}}$ denotes flexible $2$-categories \cite{Cam26}. $\mathsf{Top}_{\mathrm{cgwh}}$ denotes compactly generated weak Hausdorff spaces. $R$ is a commutative ring, and $\mathsf{Ch}(R)$ denotes unbounded chain complexes with tensor product $\otimes_R$. The Hurewicz weak equivalences are ordinary homotopy equivalences for spaces and chain-homotopy equivalences for complexes. The hypotheses in \ref{sec3assumption} are met by applying \cite[6.12]{BMR14} for spaces and \cite[6.8]{Mos19} for chain complexes. $\mathrm{Sm}(S)$ denotes smooth schemes over $S$. $\Theta_m$ is Joyal's disk category \cite{Joy97}, \cite[1.1]{Rez10}; $\Theta_0=1$, $\Theta_1=\Delta$ and $\Theta_\omega=\bigcup_m\Theta_m$. The fibrant objects of $\Theta_m\mathsf{Sp}_k$ are Rezk's $(m+k,m)$-$\Theta$-spaces, and those of $\Theta_m\mathsf{Sp}_{\infty}$ are $(\infty,m)$-$\Theta$-spaces \cite[1.2]{Rez10}. The model structure $\Theta_{\omega}\mathsf{Sp}_{\infty}$ is obtained similarly by localizing the injective model structure on $\mathsf{sSet}^{\Theta_\omega^{\mathrm{op}}}$. Next, $\mathsf{Sp}^{\Sigma}$ denotes symmetric spectra in pointed simplicial sets, with smash product $\wedge$, while $\mathsf{Mod}_E$ denotes modules over a commutative symmetric ring spectrum $E$, with relative smash product $\wedge_E$. Both use the non-positive stable model structures. $\Gamma\mS$ denotes reduced $\Gamma$-spaces in simplicial sets, with Lydakis' smash product \cite{Sch99}.

\subsubsection{Formal category theory}

The origins of the subject are readily traced back to Gray \cite{Gra66, Gra74} and Street \cite{Str72, Str74, Str74b, Str80a, Str80b}, whose paper on "the formal theory of monads" started the subject and gave it its name, by Lack's historical account \cite{Lac06}. The early days were followed by Street and Walters' work on Yoneda structures \cite{SW78}, which in turn influenced Wood who introduced proarrow equipments \cite{Woo82, Woo85}. We will say a few words about the latter. 

We see equipments as machines allowing us to perform much of category theory in a formal way, see \cite{Woo82,Woo85,Lac10,Kou14,Kou24,AM26}. There is an equipment of strictly enriched category theory, and one of internal categories. Examples of equipments have been found recently in higher categorical contexts by Riehl and Verity and Ruit \cite{RV22,Rui23}. A difference worth remarking is that Ruit works within double Segal spaces (for applications to internal $\infty$-categories), while Riehl and Verity work with virtual double categories. Our work is closely related to Riehl and Verity's and we will say more about it further on. For the relation to $\infty$-equipments, see \Cref{sec:inftyequip}.

While we defer the definitions to \textbf{\Cref{sec:formal-category-theory-equipments}}, for now we describe the elegance of the one object example: an equipment with one object consists of a fully faithful monoidal functor $\mK \hookrightarrow \mM$ with the property that every object of $\mK$ becomes right dualizable in $\mM$. 

The prototypical equipment is $\Cat$; there is a bicategory whose objects are small categories, and $\Hom(\mA , \mB)$ is given by the category of profunctors, also called modules or correspondences, from $\mA$ to $\mB$; these are defined as presheaves $\mA \times \mB^{\mathrm{op}} \to \Set$. Composition is given by coend. Any functor $f \colon \mA \to \mB$ can be seen as the profunctor $(a,b) \mapsto \operatorname{Hom}(b,f(a))$, and every functor thereby becomes dualizable in the sense that it acquires a right adjoint within the bicategory of profunctors, given by $(a,b) \mapsto \Hom\bigl(f(a),b\bigr)$.

\subsubsection{Applications of FCT} There are at least two practical uses of formal category theory. Many statements (such as 'the adjoint functor theorem', or 'left adjoints preserve colimits') can be proven formally once and for all, and then interpreted in each 'type of category theory'. Just as important is the second use: a statement expressed in the common language may be proven leveraging on the special features of one model, and then transported to an equivalent one:

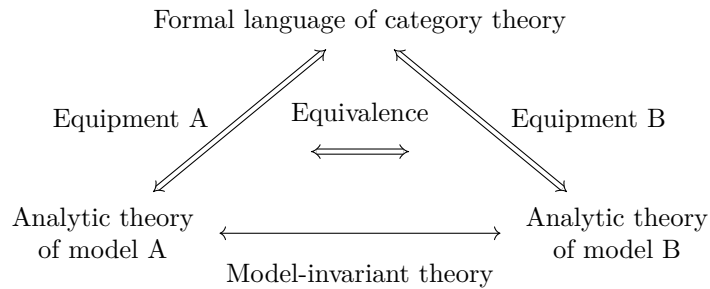
\usetikzlibrary{arrows.meta}

\begin{center}
\begin{tikzpicture}[yscale=0.8,
  every node/.style={align=center},
  equipment/.style={
    double,
    double equal sign distance,
    {Implies}-{Implies},
    shorten <=4pt,
    shorten >=4pt
  }
]
\node (formal) at (0,3.5)
  {Formal language of category theory};

\node (model) at (-3.4,0)
  {Analytic theory\\of  model A};

\node (equivalent) at (3.4,0)
  {Analytic theory\\of model B};

\draw[equipment]
  (formal) -- node[left=8pt] {Equipment A} (model);

\draw[equipment]
  (formal) -- node[right=8pt] {Equipment B} (equivalent);

\draw[<->, shorten <=6pt, shorten >=6pt]
  (model) -- node[below=8pt]
  {\small \textbf{Transfer of invariant statements}}
  (equivalent);

\draw[
  double,
  double equal sign distance,
  {Implies}-{Implies}
]
  (-0.65,1.35) -- node[above=6pt]
  {Equivalence}
  (0.65,1.35);
\end{tikzpicture}
\end{center}

It is therefore crucial to be able to compare two different semantics of category theory, so that we can tell when they are equivalent. Riehl and Verity formalize this within Makkai's \emph{first-order logic with dependent sorts} (FOLDS) \cite{Mak95}, specifying a language $\mL$ suitable for their type of equipments \cite[11.3.2 and 11.3.7]{RV22}. Makkai's invariance theorem ensures that $\mL$-equivalent $\mL$-structures satisfy the same formulas in a given context \cite[11.2.25 and 11.2.28]{RV22}. Verdugo addresses the same invariance problem using a different FOLDS language \cite[7.10--7.12]{Ver25}. Both approaches use Makkai's invariance theorem. See \cite{RV22} and \cite{Ver25} for more details, and to section \ref{subsec:equivalences}.

\subsubsection{Overview and results}
\noindent \newline 

\emph{In \textbf{\Cref{th:equipment-M}}, we prove that if $\mV$ is a monoidal model category (satisfying \ref{sec3assumption}), there is a formal category theory exhibited by an equipment whose objects are those $\mV$-categories $\mC$ such that $\mC(x,y)$ is cofibrant-fibrant in $\mV$ for every $x,y \in \mC$.} \\

We have specified what the objects are, but the important thing to tell the reader is what the corresponding notion of profunctor is in our construction. There is in fact more than one option. A natural pick, suggested by the example of $\Cat$, is to consider $\mV$-enriched presheaves $\mA \otimes \mB^{\mathrm{op}} \to \mV$ up to pointwise weak equivalence. These are well known in the dg-context \cite{Kel94}. The other choice, which is key to our proof, is given by what we call \emph{Quillen profunctors} $\mA \nrightarrow \mB$ (which we introduce in \ref{def:representable-Quillen-M}). These are left Quillen $\mV$-functors between the projective model categories $\operatorname{Psh}_{\mV}\bigl(\mA\bigr) \to \operatorname{Psh}_{\mV}\bigl(\mB\bigr)$, considered up to what we call \emph{cofibrant-wise equivalence} (see \ref{def:cofibrant-wise-equivalence-M}). The \emph{weak $\mV$-functors} between $\mV$-categories $\mA \to \mB$ are defined accordingly as those left Quillen functors $\Psh_{\mV}(\mA) \to \Psh_{\mV}(\mB)$ whose restriction on the Yoneda embedding takes values, up to isomorphism in $\Ho \Psh_{\mV}(\mB)$, within the subcategory of representables. In particular, a weak $\mV$-functor is represented by a strict one in a larger codomain.

An equivalent formulation of the existence of the equipment is given in \Cref{prop:equipment-implies-reconstruction-M}. In \Cref{prop:closed-equipment-M} we prove that the equipment admits internal Homs (left and right). These are the contents of \textbf{\Cref{sec:equipment-enriched-infinity-categories}}. 

In \textbf{\Cref{sec:inftyequip}} we show that our equipment is a quotient of an $\infty$-categorical equipment.

It is important to study how categorical notions formulated internally to the equipment translate back to enriched categories; this is what we do in \textbf{\Cref{sec:weakcategorytheory}}. We begin with the formal definition of colimit in an equipment, and prove in \Cref{sec:hocolim} that it recovers Lack-Rosicky's definition of homotopy colimit in an enriched category \cite{LR16}. We also discuss within our framework the fact that strict weighted colimits can be used to compute homotopy colimits \ref{cor:cofibrant-weight-computes-colimit-M}, and we give formulas in \Cref{sec:computecolim}. 

The finality of an arrow and its characterization are formulable in a general equipment. In \Cref{sec:realization} we define, in terms of our structure, the \emph{classifying homotopy type} $$\mathbb{B}_{\mV}(\mC) \in \Ho\mV$$ of a $\mV$-category $\mC$, and more generally the \emph{$\mV$-realization} $\lVert H \rVert_{\mV}$ of a profunctor $H$. With this invariant we study finality as follows, assuming the monoidal unit is terminal in $\mV$: \\

\emph{\textbf{\Cref{th:enriched-Quillen-A-M}.} Let $$F:\mA\to\mC$$ be a weak $\mV$-functor, and assume $\lVert\mC(c,F-)\rVert_{\mV}$ is contractible in $\mV$ for every $c \in \mC$. Then for every weak $\mV$-functor $$D:\mC\to\mE$$ if either $\colim(DF)$ or $\colim D$ exists the other does too, and the comparison map $$\colim(DF)\to \colim D$$ is invertible. Moreover, the map $$\mathbb B_{\mV}F: \mathbb B_{\mV}\mA \to \mathbb B_{\mV}\mC$$ is an isomorphism in $\Ho\mV$.} \\

This admits a three-line proof and recovers the classical Quillen's Theorem A \cite{Qui73} (as well as Joyal's quasi-categorical version \cite[\href{https://kerodon.net/tag/02NY}{02NY}]{Ker}). In \Cref{sec:comutationsclassifying}, we compute $\mathbb{B}_{\mV}$ for various bases of enrichment.

We continue our study with adjunctions, in section \ref{sec:adj}. Finally:\\

\emph{\textbf{\Cref{cor:DK-equivalence-vertical-M}} characterizes Dwyer-Kan equivalences of $\mV$-categories as the $2$-categorical equivalences in the vertical fragment of our equipment $\mV\Cat^h$.}\\

\textbf{\Cref{sec:RV}} is dedicated to proving the equivalence of our simplicial equipment with Riehl and Verity's virtual equipment of quasi-categories. 

In \cite{RV22}, Riehl and Verity show that most models of $\bigl(\infty,1\bigr)$-categories admit a formal category theory encoded in a virtual equipment, a generalization of equipments that does not require horizontal composites to exist but is still expressive enough for many concepts; they prove that the virtual equipments of these models are all equivalent \cite[Theorem~11.1.6]{RV22}, and thus establish a precise dictionary between the associated formal theories. They develop their new foundations of $\bigl(\infty,1\bigr)$-category theory internally to these virtual equipments. The models that fall within the range of applicability of their technology are quasi-categories, complete Segal spaces, Segal categories, and $1$-complicial sets. Many examples of $(\infty,1)$-categories arise in nature as topological or simplicially enriched categories: for instance, all the examples coming from simplicial model categories. We extend the scope of Riehl-Verity's foundations by adding these two models to the picture, proving the following:\\

\emph{\textbf{Theorem. (\ref{thm:comparison-before-strictification},\ref{cor:modelinvariance})} The Riehl-Verity equipment of quasicategories, the equipment of Kan-Quillen enriched simplicial categories, as well as the equipment of Quillen-Serre enriched topological categories, are all equivalent semantic interpretations of the language of formal category theory.}\\

We remark that the equipment of topological categories with respect to Hurewicz-Str\o m model structure is \emph{not} equivalent to the above ones. What it models is a refinement of $(\infty,1)$-category theory where mapping spaces are allowed to be arbitrary topological spaces considered up to classical homotopy equivalence. 

Indeed, for enriching bases $\mV,\mN$ with fibrant monoidal unit, a necessary condition for equivalence of the associated equipments is monoidal equivalence of their homotopy categories $\Ho \mV$ and $\Ho \mN$. In the other direction, at least for nice enough model categories (for example combinatorial and satisfying the monoid axiom), a symmetric monoidal Quillen equivalence $\mV \to \mN$ induces an equivalence of equipments $\mathbb E_{\mV} \to \mathbb E_{\mN}$.

\subsubsection{Timeline and related work}
Weakly enriched categories, also known as enriched $\infty$-categories \cite{GH15}, began to emerge in the 1960s, alongside their strict counterparts. In fact, three of the earliest examples of enriched categories, namely differential graded categories \cite{Kel64,Kel65}, topological categories \cite{EM45}, and $2$-categories \cite{Ben65,Mar65}, all arise in contexts where their hom-objects are most naturally considered up to some notion of weak equivalence other than isomorphism: quasi-isomorphism, weak homotopy equivalence, and equivalence of categories.

In the second half of the 1960s Rainer Vogt, reporting on his joint work with Boardman, records in his PhD thesis \cite{Vog68} a composition operation defined "up to a homotopy, which is itself defined only up to a homotopy, which is itself defined only up to
a homotopy, which is ... ". 
Around the same time, in 1967, B\'enabou introduces bicategories in the Reports of the Midwest Category Seminar \cite{Ben67}, describing them in the opening lines as enriched "up to given coherent isomorphisms".\\ While the Australian school of category theorists began studying the relationship between bicategories and $2$-categories with Street \cite{Str80b}, in the 1980s Dwyer and Kan started to look at simplicial enrichment in order to provide an enhanced version of localization that would carry information in the higher homotopy groups of the hom-objects \cite{DK80,DK80a}; in the same decade Cordier and Porter contributed to founding homotopy-coherent category theory with simplicial categories \cite{CP97}. Grothendieck was also thinking about higher categories in the 1980s, while writing his influential manuscript \cite{Gro83}, where he formulated the homotopy hypothesis. His approach was later pushed further by Maltsiniotis \cite{Mal10}.

From the 1990s, enrichment in linear homotopy types has become relevant in geometry, after Fukaya's $A_{\infty}$-categories \cite{Fuk93} (following Stasheff's $A_{\infty}$-algebras), in relation to Kontsevich's homological mirror symmetry, and the work of Bondal-Kapranov, Keller and Drinfeld in the strict dg-world \cite{BK90,Kel94,Dri04}. In particular, the bicategory of dg-categories and quasi-functors goes back to Keller \cite{Kel94}. Drinfeld \cite[16.6]{Dri04} proves its locally fully faithful inclusion into derived bimodules and the existence of conjoints, thereby producing an equipment. The latter dg-equipment has been recently constructed and applied by Imamura in \cite{Ima24}. Our proof is different, since we need model-categorical methods allowing for the generality of our assumptions in $\ref{sec3assumption}$. It can be shown by \Cref{lem:homotopical-Yoneda-comparison-M} that the two equipments are equivalent. Going back to the history of the differential graded setting, Tabuada and To\"en studied the homotopy theory of dg-categories \cite{Tab05, Toe07}. 

On the homotopy theory of $\mV$-categories, we mention also the work of Lack on the $2$-categorical case \cite{Lac02}; Bergner (following Dwyer and Kan) took care of the simplicial one \cite{Ber07}, while Lurie, Berger-Moerdijk and Muro extended this development to a general base \cite{HTT,BM13,Mur15}. 

Around the same years, in motivic geometry, enriched methods were deployed in the study of generalized cohomology theories, after Dundas, Röndigs and Østvær \cite{DRO03,DRO03b}.

Higher categories are a particular case with a history on its own; in addition to the work of the Australian school with the tricategories of Gordon-Power-Street \cite{GPS95} and the complicial approach of Roberts, Street and Verity \cite{Str87,Str03,Ver08}, we must mention Batanin \cite{Bat98}, Batanin-Cisinski-Weber \cite{BCW13}, Leinster and Trimble \cite{Lei02}, the combinatorial models of Tamsamani, Barwick, Ara, Rezk, Paoli and Ozornova-Rovelli \cite{Tam99, Bar05, Ara14, Rez01,Rez10,Pao19,OR20}, and the cubical approach \cite{DKLS24, CKM25}. Weighted (co)limits in $(\infty,n)$-categories are studied by Moser-Rovelli-Rasekh \cite{MRR26,MRR24}. Our approach, while coming from a different angle, is compatible with theirs, given the results in \cite{MRR26}. 

Enriched Segal categories were defined by Pellissier and Simpson \cite{Pel02,Sim12}. More recently in 2010 Bacard extended the approach to non-cartesian bases \cite{Bac10}. Later, Lowen and Mertens \cite{LM25} introduced enriched quasicategories. 

Internally to Joyal's quasicategory theory \cite{Joy02}, and within the broader context of Lurie's $\infty$-operads and higher algebra, various definitions of enrichment were introduced by Lurie and Gepner-Haugseng \cite{Lur07,GH15}, followed by Hinich \cite{Hin20}, who also treats weighted colimits \cite{Hin23}, while bi-enrichment is due to Heine \cite{Hei24}, with an account of weighted colimits \cite{Hei24b} within Lurie's definition of enrichment. Our approach is independent to Heine's, as well as to Hinich's version \cite{Hin23}, and the compatibility with those frameworks will be shown in forthcoming work. Still within the framing of Lurie's higher algebra, another definition in terms of marked modules is proposed by Reutter-Zetto in \cite{RZ25}. The homotopy theory of $\mV$-categories enriched in a nice enough model structure is shown by Haugseng's \cite{Hau15} to be equivalent to the one of enriched $\infty$-categories in the sense of \cite{GH15}.

\subsubsection{Acknowledgements}
This paper is written as part of a larger project that will be the author's PhD thesis. My deep gratitude goes to my PhD advisors, Paul Arne Østvær and Jiří Rosický.

I thank the organizers and scientific committees of CT24 in Santiago and CT26 in Baltimore for giving me the opportunity to speak about this project. I also wish to thank Emily Riehl for explaining to me her view of model-independent $(\infty,1)$-category during a conversation at CT25.

The support of GACR through the grant GA22-02964S, \emph{Enriched categories and their application}, as well as of the project MUNI/A/1648/2025, is gratefully acknowledged.

\section{Equipments}
\label{sec:formal-category-theory-equipments}

We begin introducing Wood's definition of equipment \cite{Woo82}. Through this section \emph{bicategory} is a synonym for \emph{weak $2$-category}, whereas $2$\emph{-category} is a shorthand for a category strictly enriched in the $1$-category of categories $\bigl(\Cat,\times,1 \bigr)$.

\begin{definition}\label{def:woodequip}
An \emph{equipment} on a $2$-category $\mK$ comprises the following structure: a bicategory $\mM$, sharing the same objects, together with a pseudofunctor $$(-)_*:\mK\to \mM$$ satisfying the following properties:
\begin{itemize}
\item[(i)] $(-)_*$ is the identity on objects;
\item[(ii)] for all objects $A,B$, the functor $$(-)_*:\mK(A,B)\to \mM(A,B)$$ is fully faithful;
\item[(iii)] for every arrow $f$ in $\mK$, the arrow $f_*$ has a right adjoint $f^\vee$ in the bicategory $\mM$.
\end{itemize}
\end{definition}
In Wood's original definition, not only $\mM$ but also $\mK$ is allowed to be a bicategory rather than a $2$-category, but we will not need such a generality and in fact, in all except one examples in this paper, $\mM$ will be a $2$-category as well.

\subsection{Companions and conjoints} We shall use the double-categorical manifestation of equipments introduced by Verity \cite{Ver92}. \subsubsection{Pseudo double categories} First, recall that a \emph{pseudo double category} is a structure comprising objects, vertical arrows, horizontal arrows, and squares $$\begin{tikzcd}[column sep=3em,row sep=2.4em] A \ar[r,"X"] \ar[d,"u"'] & B \ar[d,"v"] \\ C \ar[r,"Y"'] & D \arrow[phantom,from=1-1,to=2-2,"\Downarrow\varphi" description] \end{tikzcd}$$ with vertical arrows $u,v$ and horizontal arrows $X,Y$. We use $\to$ for vertical arrows and $\nrightarrow$ for horizontal arrows; while the former compose strictly, and squares also compose vertically strictly, the horizontal arrows and squares compose horizontally in a way associative and unital only up to coherent isomorphism. The horizontal and vertical compositions of squares in a pseudo double category must satisfy the \emph{interchange law}: given four composable squares as in $$\begin{tikzcd}[column sep=3em,row sep=2.4em] \bullet \ar[r] \ar[d] & \bullet \ar[r] \ar[d] & \bullet \ar[d] \\ \bullet \ar[r] \ar[d] & \bullet \ar[r] \ar[d] & \bullet \ar[d] \\ \bullet \ar[r] & \bullet \ar[r] & \bullet \arrow[phantom,from=1-1,to=2-2,"\Downarrow\alpha" description] \arrow[phantom,from=1-2,to=2-3,"\Downarrow\beta" description] \arrow[phantom,from=2-1,to=3-2,"\Downarrow\gamma" description] \arrow[phantom,from=2-2,to=3-3,"\Downarrow\delta" description] \end{tikzcd}$$ one can first paste horizontally the squares in the top and bottom halves and then compose vertically the results, or first compose vertically the left and right halves and then compose horizontally the results.

\begin{notation}
If $\bD$ is a pseudo double category, we obtain a $2$-category by remembering only its \emph{vertical fragment}, and a bicategory by remembering only its \emph{horizontal fragment}. More precisely, let $V(\bD)^{\mathrm{co}}$ be the $2$-category whose objects are those of $\bD$, whose arrows are the vertical arrows of $\bD$, and whose $2$-cells $u\Rightarrow v$ are squares in $\bD$ whose top and bottom are horizontal identities: $$\begin{tikzcd}[column sep=3em,row sep=2.4em] A \ar[r,equal] \ar[d,"v"'] & A \ar[d,"u"] \\ B \ar[r,equal] & B \arrow[phantom,from=1-1,to=2-2,"\Downarrow\theta" description] \end{tikzcd}$$ We let $H(\bD)$ be the bicategory whose objects are those of $\bD$, whose arrows are the horizontal arrows of $\bD$, and whose $2$-cells are squares $$\begin{tikzcd}[column sep=3em,row sep=2.4em] A \ar[r,"J"] \ar[d,equal] & B \ar[d,equal] \\ A \ar[r,"K"'] & B \arrow[phantom,from=1-1,to=2-2,"\Downarrow\alpha" description] \end{tikzcd}$$
\end{notation}

\subsubsection{Companions and conjoints in a double category}
\begin{notation}
    We write $P\coheq Q$ when we omit the horizontal associators and unitors.
\end{notation}
\begin{definition}
Let $u:A\to B$ be a vertical arrow in a pseudo double category $\bD$. A \emph{companion} of $u$ is a horizontal arrow $$u_*:A\nrightarrow B$$ such that there exist squares $$\begin{tikzcd}[column sep=3em,row sep=2.4em] A \ar[r,"u_*"] \ar[d,"u"'] & B \ar[d,equal] \\ B \ar[r,equal] & B \arrow[phantom,from=1-1,to=2-2,"\Downarrow\alpha_u" description] \end{tikzcd} \qquad \begin{tikzcd}[column sep=3em,row sep=2.4em] A \ar[r,equal] \ar[d,equal] & A \ar[d,"u"] \\ A \ar[r,"u_*"'] & B \arrow[phantom,from=1-1,to=2-2,"\Downarrow\beta_u" description] \end{tikzcd}$$ where $$\begin{tikzcd}[column sep=3em,row sep=2.4em] A \ar[r,equal] \ar[d,equal] & A \ar[r,"u_*"] \ar[d,"u"] & B \ar[d,equal] \\ A \ar[r,"u_*"'] & B \ar[r,equal] & B \arrow[phantom,from=1-1,to=2-2,"" description] \arrow[phantom,from=1-2,to=2-3,"" description] \end{tikzcd} \coheq \begin{tikzcd}[column sep=3em,row sep=2.4em] A \ar[r,"u_*"] \ar[d,equal] & B \ar[d,equal] \\ A \ar[r,"u_*"'] & B \arrow[phantom,from=1-1,to=2-2,"\Downarrow \id" description] \end{tikzcd}$$ and $$\begin{tikzcd}[column sep=3em,row sep=2.4em] A \ar[r,equal] \ar[d,equal] & A \ar[d,"u"] \\ A \ar[r,"u_*"'] \ar[d,"u"'] & B \ar[d,equal] \\ B \ar[r,equal] & B \arrow[phantom,from=1-1,to=2-2,"" description] \arrow[phantom,from=2-1,to=3-2,"" description] \end{tikzcd} = \begin{tikzcd}[column sep=3em,row sep=2.4em] A \ar[r,equal] \ar[d,"u"'] & A \ar[d,"u"] \\ B \ar[r,equal] & B \arrow[phantom,from=1-1,to=2-2,"\Downarrow\id" description] \end{tikzcd}.$$
\end{definition}

\begin{definition}
Let $u:A\to B$ be a vertical arrow in a pseudo double category $\bD$. A \emph{conjoint} of $u$ is a horizontal arrow $$u^\vee:B\nrightarrow A$$ such that there are counit and unit squares $$\begin{tikzcd}[column sep=3em,row sep=2.4em] B \ar[r,"u^\vee"] \ar[d,equal] & A \ar[d,"u"] \\ B \ar[r,equal] & B \arrow[phantom,from=1-1,to=2-2,"\Downarrow\varepsilon_u" description] \end{tikzcd} \qquad \begin{tikzcd}[column sep=3em,row sep=2.4em] A \ar[r,equal] \ar[d,"u"'] & A \ar[d,equal] \\ B \ar[r,"u^\vee"'] & A \arrow[phantom,from=1-1,to=2-2,"\Downarrow\eta_u" description] \end{tikzcd}$$ and the triangle identities hold: $$\begin{tikzcd}[column sep=3em,row sep=2.4em] B \ar[r,"u^\vee"] \ar[d,equal] & A \ar[r,equal] \ar[d,"u"] & A \ar[d,equal] \\ B \ar[r,equal] & B \ar[r,"u^\vee"'] & A \arrow[phantom,from=1-1,to=2-2,"" description] \arrow[phantom,from=1-2,to=2-3,"" description] \end{tikzcd} \coheq \begin{tikzcd}[column sep=3em,row sep=2.4em] B \ar[r,"u^\vee"] \ar[d,equal] & A \ar[d,equal] \\ B \ar[r,"u^\vee"'] & A \arrow[phantom,from=1-1,to=2-2,"\Downarrow \id" description] \end{tikzcd}$$

$$\begin{tikzcd}[column sep=3em,row sep=2.4em] A \ar[r,equal] \ar[d,"u"'] & A \ar[d,equal] \\ B \ar[r,"u^\vee"'] \ar[d,equal] & A \ar[d,"u"] \\ B \ar[r,equal] & B \arrow[phantom,from=1-1,to=2-2,"" description] \arrow[phantom,from=2-1,to=3-2,"" description] \end{tikzcd} = \begin{tikzcd}[column sep=3em,row sep=2.4em] A \ar[r,equal] \ar[d,"u"'] & A \ar[d,"u"] \\ B \ar[r,equal] & B \arrow[phantom,from=1-1,to=2-2,"\Downarrow\id" description] \end{tikzcd}.$$
\end{definition}

\subsubsection{From equipments to equipments} Given $(-)_*:\mK\to\mM$ an equipment in the sense of definition \ref{def:woodequip}, there is a pseudo double category $\bD$ such that $V(\bD)^{\mathrm{co}}\cong\mK$ and $H(\bD)\cong\mM$. The unit and counit of $u_*\dashv u^\vee,$ together with the triangle identities, satisfy the definition for a conjoint square. Conversely, let $\bD$ be a pseudo double category. The (unique up to contractible groupoid) choice of companion and conjoint for every vertical arrow $u$ determines a pseudofunctor $(-)_*:V(\bD)^{\mathrm{co}}\to H(\bD) .$ For vertical arrows $u,v:A\to B$ there is a bijection $V(\bD)^{\mathrm{co}}(u,v) \cong H(\bD)(u_*,v_*),$ given by $$\begin{tikzcd}[column sep=3em,row sep=2.4em] A \ar[r,equal] \ar[d,"v"'] & A \ar[d,"u"] \\ B \ar[r,equal] & B \arrow[phantom,from=1-1,to=2-2,"\Downarrow\theta" description] \end{tikzcd} \quad\to\quad \begin{tikzcd}[column sep=3em,row sep=2.4em] A \ar[r,"u_*"] \ar[d,equal] & B \ar[d,equal] \\ A \ar[r,"v_*"'] & B \arrow[phantom,from=1-1,to=2-2,"\Downarrow \theta_*" description] \end{tikzcd}$$ where $$\begin{tikzcd}[column sep=3em,row sep=2.4em] A \ar[r,"u_*"] \ar[d,equal] & B \ar[d,equal] \\ A \ar[r,"v_*"'] & B \arrow[phantom,from=1-1,to=2-2,"\Downarrow \theta_*" description] \end{tikzcd} :\coheq \begin{tikzcd}[column sep=3em,row sep=2.4em] A \ar[r,equal] \ar[d,equal] & A \ar[r,equal] \ar[d,"v"] & A \ar[r,"u_*"] \ar[d,"u"] & B \ar[d,equal] \\ A \ar[r,"v_*"'] & B \ar[r,equal] & B \ar[r,equal] & B \arrow[phantom,from=1-1,to=2-2,"\Downarrow\beta_v" description] \arrow[phantom,from=1-2,to=2-3,"\Downarrow\theta" description] \arrow[phantom,from=1-3,to=2-4,"\Downarrow\alpha_u" description] \end{tikzcd}$$ Finally, the conjoint squares exhibit the adjunction $u_*\dashv u^\vee$.

\subsection{Equivalences}\label{subsec:equivalences}
A \emph{morphism of equipments} is a double pseudofunctor, pseudo in both directions, following Shulman's definition \cite[6.1]{Shu11}. Such morphisms preserve companions, conjoints and mate correspondences \cite[6.7 and 6.9]{Shu11}.

\begin{definition}\label{def:equipequivalence}
A morphism $F:\mathbb D\to\mathbb E$ is an \emph{equipment biequivalence} if:
\begin{itemize}
\item[$(w1)$] Every object of $\mathbb E$ is vertically equivalent to some $FA$.
\item[$(w2)$] Every vertical arrow $FA\to FB$ is isomorphic to $Fu$ for some vertical $u:A\to B$.
\item[$(w3')$] Every horizontal arrow $FA\nrightarrow FB$ is isomorphic to $FJ$ for some horizontal $J:A\nrightarrow B$.
\item[$(w4)$] The action on squares with any prescribed boundary is bijective.
\end{itemize}
\end{definition}

Isomorphisms in $(w2)$ and $(w3')$ are taken in the respective vertical and horizontal fragments. By the companion correspondence, these conditions are equivalent to asking that both induced pseudofunctors on the vertical and horizontal fragments are biequivalences.

\Cref{def:equipequivalence} agrees with Riehl-Verity's virtual-biequivalence criterion \cite[Theorem 11.1.6]{RV22}.

For strict morphisms between strict equipments, \ref{def:equipequivalence}  is Verdugo's criterion \cite[4.12 and 4.16]{Ver25}, following Moser-Sarazola-Verdugo \cite{MSV22} with the directions exchanged \cite[Remark~4.14]{Ver25}. On strict equipments it agrees with Campbell's \emph{gregarious double equivalences} \cite[5.5]{CL26}. The latter were advocated as the ``correct'' notion of equivalence of double categories by Sarazola in her lecture at the 2025 Category Theory Conference in Brno \cite{Sar25}.

\begin{proposition}\label{prop:RVstructureequiv}
Let $\mL_{RV}$ be Riehl-Verity's language \cite[11.3.2]{RV22}. An equipment biequivalence $F \colon \mathbb D \to \mathbb E$ induces an equivalence of $\mL_{RV}$-structures \cite[11.2.25 and 11.3.7]{RV22}.
\end{proposition}

\begin{proof}
Let $\mathbb{E}^{\operatorname{str}}  \to \mathbb{E}$ and $\mathbb{D}^{\operatorname{str}}  \to \mathbb{D}$ be the horizontal strictifications \cite{GP99,Cam19} of $\mathbb E$ and $\mathbb D$. Both these equivalences induce equivalences of $\mL_{RV}$-structures by construction. We now have an equipment equivalence $\mathbb{D}^{\operatorname{str}}  \to \mathbb{D} \to \mathbb{E} \to  \mathbb E^{\operatorname{str}} $ between strict equipments, and we can therefore apply Campbell-Leinster's \cite[5.8]{CL26} which provides us (in their terminology) with a span of strict surjective equivalences between strict equipments $\mathbb{D}^{\operatorname{str}}  \leftarrow \mathbb{C} \rightarrow  \mathbb{E}^{\operatorname{str}} $, which satisfies Makkai's definition of equivalence of $\mL$-structures \cite[11.2.25]{RV22}.
\end{proof}

\begin{proposition}\label{Verdugostructureequiv}
Let $\mL_{V}$ be Verdugo's language \cite[7.10]{Ver25}. If $\mathbb{D}$ and $\mathbb E$ are both strict double categories, an equipment biequivalence $F \colon \mathbb D \to \mathbb E$ induces an equivalence of $\mL_{V}$-structures \cite[6.31]{Ver25}.
\end{proposition}
\begin{proof}
   Campbell-Leinster's \cite[5.8]{CL26} again provides us with a span of strict surjective equivalences between strict equipments and therefore the conclusion follows in the same way as above by Makkai's definition of equivalence of $\mL$-structures.
\end{proof}
\section{The equipment of weak category theory}
\label{sec:equipment-enriched-infinity-categories}
\subsubsection{Assumptions on $\mV$} \label{sec3assumption} Throughout this paper, $\mV$ denotes a symmetric monoidal model category with cofibrant unit. We call a $\mV$-category \emph{locally cofibrant} (respectively \emph{locally fibrant}) when every hom-object $\mC(x,y)$ is cofibrant (respectively fibrant) in $\mV$. We further assume that for every locally cofibrant small $\mV$-category $\mC$, the $\mV$-functor category $\mV \Fun(\mC,\mV)$ admits the projective model structure; very general existence theorems for the enriched projective model structure are established in \cite{DRO03}[4.2], \cite{Shu06}[24.4], and \cite{Mos19}[4.4; 6.5], covering all examples of interest.
\subsubsection{Some lemmas}\label{equipmentprelim} 
We recall that a \emph{$\mV$-model category} $\mathcal N$ is a $\mV$-enriched, $\mV$-tensored, and $\mV$-cotensored category, equipped with a model structure (on its underlying category), such that: for every cofibration $i:K\to L$ in $\mV$ and every cofibration $j:X\to Y$ in $\mN$, the pushout-product map $$(L\otimes X)\coprod_{K\otimes X}(K\otimes Y) \to L\otimes Y$$ is a cofibration in $\mN$, and it is a trivial cofibration if either $i$ or $j$ is a weak equivalence. We write $\Psh_{\mV}(\mA) = \mV\Fun(\mA^{\op},\mV)$ and we assume it endowed with the projective model structure.

\begin{lemma}
\label{lem:evaluation-both-quillen-M}
Let $\mA$ be a small and locally cofibrant $\mV$-category, and let $\mathcal N$ be a $\mV$-model category. For every object $a\in\mA$, the evaluation functor $$ev_a: \mV\Fun(\mA,\mN)_{\mathrm{proj}} \to \mN$$ is both left and right Quillen.
\end{lemma}

\begin{proof}
 By definition of projective (trivial) fibrations $ev_a$ is right Quillen. To prove it is left Quillen, we show that its right adjoint is also right Quillen. Such right adjoint sends $Y\in\mathcal N$ to the $\mV$-functor whose value at $b$ is given by the cotensor $Y^{\mA(b,a)}.$ By assumption every $\mA(b,a)$ is cofibrant; hence the fact that $\mathcal N$ is a $\mV$-model category guarantees that cotensoring with $\mA(b,a)$ preserves fibrations and trivial fibrations.
\end{proof}

\begin{corollary}
\label{cor:projectively-cofibrant-objectwise-cofibrant-M}
Under the hypotheses of \Cref{lem:evaluation-both-quillen-M}, every projectively cofibrant enriched diagram is objectwise cofibrant.
\end{corollary}

\begin{lemma}[]
\label{lem:cofibrant-weight-homotopy-invariance}
Under the hypotheses of \Cref{lem:evaluation-both-quillen-M}, let $P\in\Psh_{\mV}(\mA)$ be projectively cofibrant. If $ u:X\to Y $ is an objectwise weak equivalence between objectwise cofibrant enriched functors $X,Y:\mA\to\mN,$ then $P\otimes_{\mA}X \to P\otimes_{\mA}Y$ is a weak equivalence in $\mN$.
\end{lemma}

\begin{proof}
By a retract argument we reduce to the case where $ i:X\to Y$ is an objectwise trivial cofibration. Then $ P\otimes_{\mA}i: P\otimes_{\mA}X\to P\otimes_{\mA}Y$ is a trivial cofibration in $\mN$ for every cofibrant $P$, because $\mN$ is a $\mV$-model category.
\end{proof}

\begin{lemma}
\label{lem:diagram-derived-adjunction-M}
Let $ F:\Psh_{\mV}(\mA) \rightleftarrows \Psh_{\mV}(\mB):G $ be a $\mV$-enriched Quillen adjunction. Let $\mI$ be a small locally cofibrant $\mV$-category. Then postcomposition gives a Quillen adjunction $$  F_*:\mV\Fun(\mI,\Psh_{\mV}(\mA))_{\mathrm{proj}} \rightleftarrows \mV\Fun(\mI,\Psh_{\mV}(\mB))_{\mathrm{proj}} :G_*.$$ In particular, if $U:\mI\to\Psh_{\mV}(\mA)$ is objectwise cofibrant and $  V:\mI\to\Psh_{\mV}(\mB)$ is objectwise fibrant, there is a natural isomorphism $$\Ho\bigl(\mV\Fun(\mI,\Psh_{\mV}(\mB))\bigr)(FU,V) \cong \Ho\bigl(\mV\Fun(\mI,\Psh_{\mV}(\mA))\bigr)(U,GV) \eqno{(\ast)}$$

\end{lemma}

\begin{proof}
Since projective fibrations and trivial fibrations are objectwise, $F_*\dashv G_*$ is Quillen. Let $QU\to U$ be a projective cofibrant replacement. By \Cref{cor:projectively-cofibrant-objectwise-cofibrant-M}, the diagram $QU$ is objectwise cofibrant. Since $U$ is cofibrant-valued, $QU\to U$ is an objectwise weak equivalence between objectwise cofibrant functors, therefore $FQU\to FU$ is an objectwise weak equivalence. Thus the left derived functor of $F_*$ composed to $U$ is just $\mathbf L F_*(U)\cong FU$, and $\mathbf R G_*(V)\cong GV$ because $V$ is projectively fibrant.

\end{proof}

\subsection{Yoneda lemma}

\begin{definition}\label{lem:leftquiliffobjcof}
   For a $\mV$-enriched diagram $X:\mA\to \mE$ with $\mE$ $\mV$-cocomplete, we denote $$(-)\otimes_{\mA}X = \Lan_{\y_{\mA}}X: \Psh_{\mV}(\mA)\to \mE$$ the $\mV$-functor whose right adjoint is $$e \mapsto \mE\bigl(X-,e\bigr)$$
   Let $\mV$ be a monoidal model category, if $\mE$ is a $\mV$-model category, then by definition the functor $(-)\otimes_{\mA}X $ is left Quillen with respect to the projective model structure if and only if $X(a)$ is cofibrant for every $a\in\mA$.\end{definition}

\begin{definition}
\label{def:cofibrant-wise-equivalence-M} \label{lem:detection-on-representables-M}
Let $F,G:\Psh_{\mV}(\mA) \to\mN$ be $\mV$-enriched left Quillen functors. We call a $\mV$-natural transformation $\alpha:F\to G$ a \emph{cofibrant-wise equivalence} if, for every cofibrant $P\in\Psh_{\mV}(\mA)$, the map $$\alpha_P:F(P)\to G(P)$$ is a weak equivalence in $\mN$. We call it a \emph{representable-wise equivalence} if, for every $a\in\mA$, the map $$\alpha_{\y_{\mA}(a)}: F(\y_{\mA}(a))\to G(\y_{\mA}(a))$$ is a weak equivalence. By \Cref{lem:cofibrant-weight-homotopy-invariance}, these two notions precisely coincide.
\end{definition}

We recall that a \emph{relative category} is a pair $\bigl(\mC,\mW\bigr)$ where $\mC$ is a category and $\mW$ is a subcategory containing all the identities; any model category with its choice of weak equivalences is in particular a relative category. 

An \emph{homotopy equivalence of relative categories} is a pair  $F,G$ of functors preserving the weak equivalences and such that both $FG$ and $GF$ are connected to the identity by a finite zig-zag of natural transformations which are objectwise weak equivalences. If $\bigl(\mC,\mW_{\mC}\bigr)$ and $\bigl(\mD,\mW_{\mD}\bigr)$ are homotopy equivalent, there is an associated equivalence $\mC\left[\mW_{\mC}^{-1}\right] \simeq \mD\left[\mW_{\mD}^{-1}\right]$ of their localizations.

We consider the relative category $\mV\Fun^{L}_{\mathrm{Quil}} \bigl( \Psh_{\mV}(\mA),\Psh_{\mV}(\mB) \bigr)$ of $\mV$-enriched left Quillen functors equipped with the cofibrant-wise equivalences.

\begin{lemma}[]
\label{lem:homotopical-Yoneda-comparison-M}
Restriction along $\y_{\mA}$ induces a homotopy equivalence of relative categories $$\mV\Fun^{L}_{\mathrm{Quil}} \bigl( \Psh_{\mV}(\mA),\Psh_{\mV}(\mB) \bigr)   \simeq   \mV\Fun(\mA,\Psh_{\mV}(\mB))_{\mathrm{proj}} \simeq \Psh_{\mV}(\mA^{\op}\otimes\mB)_{\mathrm{proj}} .$$
\end{lemma}

\begin{proof}
\label{constr:bar-M}
Let $X:\mA\to\Psh_{\mV}(\mB)$ be a $\mV$-functor. We define, up to a contractible space of choices, the $\mV$-enriched left Quillen functor $$\overline X := \Lan_{\y_{\mA}}QX \colon \Psh_{\mV}(\mA) \to \Psh_{\mV}(\mB)$$ where the contractible ambiguity is due to the choice of a projectively cofibrant replacement $QX\xrightarrow{\sim}X$ in $\mV\Fun(\mA,\Psh_{\mV}(\mB))_{\mathrm{proj}}$ ensuring that, by Corollary \ref{cor:projectively-cofibrant-objectwise-cofibrant-M}, each $QX(a)$ is cofibrant. There is a weak equivalence $\overline X\circ\y_{\mA} \cong QX \xrightarrow{\sim} X$, so that $\overline{(-)}$ is a right homotopy inverse to precomposition with $\y_{\mA}$. 

In the other direction, given $F \in \mV\Fun^{L}_{\mathrm{Quil}} \bigl( \Psh_{\mV}(\mA),\Psh_{\mV}(\mB) \bigr) $, 
the map $Q(F\y_{\mA})\to F\y_{\mA}$ induces a representable-wise weak equivalence \ref{lem:detection-on-representables-M}
$$\overline{F\y_{\mA}} = \Lan_{\y_{\mA}}Q(F\y_{\mA}) \to \Lan_{\y_{\mA}}(F\y_{\mA}) \cong F.$$
\end{proof}

\subsection{Quillen profunctors}

\begin{definition}
\label{def:Pro-M-hom}
For locally fibrant-cofibrant $\mV$-categories $\mA,\mB$, let $\Pro_{\mV}(\mA,\mB)$ be a category defined by formally inverting the cofibrant-wise equivalences in $\mV\Fun^{L}_{\mathrm{Quil}} \bigl( \Psh_{\mV}(\mA),\Psh_{\mV}(\mB) \bigr)$.
\end{definition}
Note that by \Cref{lem:homotopical-Yoneda-comparison-M} this localization is locally small. By \Cref{lem:detection-on-representables-M}, one obtains the same localization inverting representable-wise equivalences.

Composition of enriched left Quillen functors preserves cofibrant-wise equivalences, hence we obtain a $2$-category, as follows.
\begin{definition}
\label{def:Pro-M}
Let $\Pro_{\mV}$ be the $2$-category whose objects are locally fibrant-cofibrant $\mV$-categories and whose hom-object from $\mA$ to $\mB$ is the category $\Pro_{\mV}(\mA,\mB)$.
\end{definition}
\begin{remark}
    $\Pro_{\mV}(\mA,\mB)$ is the underlying category of an $\Ho\mV$-enriched category that we denote $\underline{\Pro_{\mV}}\bigl(\mA,\mB\bigr)$, so that we have $\Pro_{\mV}(\mA,\mB) = \pi_0 \underline{\Pro_{\mV}}\bigl(\mA,\mB\bigr)$ where $\operatorname{\pi_0}\coloneqq \bigl(\operatorname{Ho}\mV\bigr) \bigl(\tu,-\bigr)$. Therefore, $\Pro_{\mV}$ underlies what might be called an $\bigl(\Ho\mV,2\bigr)$-category $\underline{\Pro_{\mV}}$, by which we mean a category enriched in the category of $\Ho\mV$-enriched categories, with tensor product induced by the derived tensor product of $\mV$.
\end{remark}
\begin{definition}
\label{def:homotopy-representables-M}
Let $\mB$ be a small locally fibrant-cofibrant $\mV$-category. Define $$\widetilde{\y(\mB)} \subseteq \Psh_{\mV}(\mB)$$ to be the full $\mV$-subcategory spanned by those objects $E$ which are fibrant and cofibrant in $\Psh_{\mV}(\mB)$ and isomorphic in $\Ho\Psh_{\mV}(\mB)$ to $\y_{\mB}(b)$ for some $b\in\mB$.
\end{definition}

\begin{definition}
\label{def:representable-Quillen-M}\label{lem:representable-normal-form-M}
A \emph{Quillen $\mV$-profunctor} is an object $F\in\Pro_{\mV}(\mA,\mB)$. It is called \emph{representable} if there exists a $\mV$-functor $X:\mA\to\widetilde{\y(\mB)}$ and an isomorphism $F\cong\overline X$ in $\Pro_{\mV}(\mA,\mB)$.

If $Z:\mA\to\Psh_{\mV}(\mB)$ is a $\mV$-functor whose values are cofibrant and such that $Z(a)$ is isomorphic in $\Ho\Psh_{\mV}(\mB)$ to a representable, then there is an objectwise weak equivalence $Z\to Z^f$ in $\mV\Fun(\mA,\Psh_{\mV}(\mB))_{\mathrm{proj}}$ such that $Z^f$ takes values in $\widetilde{\y(\mB)}$, and $\overline Z\cong \overline{Z^f}$ in $\Pro_{\mV}(\mA,\mB)$.
\end{definition}

\begin{lemma}
\label{lem:representables-closed-M}
Representable Quillen $\mV$-profunctors are closed under composition, and the identity is representable.
\end{lemma}

\begin{proof}
Let $X:\mA\to\widetilde{\y(\mB)}$ and $Y:\mB\to\widetilde{\y(\mC)}$ represent two profunctors. By \Cref{lem:representable-normal-form-M} the diagram $\overline Y\circ X: \mA\to\Psh_{\mV}(\mC)$ may be replaced by a functor valued in $\widetilde{\y(\mC)}$. Moreover, by \Cref{lem:homotopical-Yoneda-comparison-M}, $\overline Y\circ\overline X \cong \overline{\overline Y\circ X}$ in $\Pro_{\mV}(\mA,\mC)$ and the Yoneda embedding represents the identity: $\overline{\y_{\mA}}\cong\Lan_{\y_{\mA}}\y_{\mA} \cong \id_{\mA}$ in $\Pro_{\mV}(\mA,\mA)$.
\end{proof}
\subsection{The construction}

\begin{definition}
\label{def:Cat-h-M}
Let $\mV\Cat^h(\mA,\mB)$ be the full subcategory of $\Pro_{\mV}(\mA,\mB)$ spanned by the representable Quillen $\mV$-profunctors. We obtain a strict $2$-category denoted $ \mV\Cat^h$.
\end{definition}

\begin{theorem}
\label{th:equipment-M}
$$\mV\Cat^h \to \Pro_{\mV}$$ is an equipment.
\end{theorem}

\begin{proof}
Let $F\in \mV\Cat^h(\mA,\mB)$, by definition given as $F=\overline X$ for some $X:\mA\to\widetilde{\y(\mB)}.$ Since $X$ is cofibrant-valued, we may take $F=\Lan_{\y_{\mA}}X$. Let $G$ be the right Quillen $\mV$-adjoint to $F$. We introduce the purported conjoint of $F$ as $$F^\vee := \overline{G\y_{\mB}} : \Psh_{\mV}(\mB)\to \Psh_{\mV}(\mA).$$ where $$\overline{(-)}\colon \mV\Fun(\mA,\Psh_{\mV}(\mB))_{\mathrm{proj}}   \to  \mV\Fun^{L}_{\mathrm{Quil}} \bigl( \Psh_{\mV}(\mA),\Psh_{\mV}(\mB) \bigr)$$ is defined in the proof of \ref{lem:homotopical-Yoneda-comparison-M}. By such a construction, there is a canonical isomorphism $$\gamma: F^\vee\y_{\mB} \xrightarrow{\cong} G\y_{\mB}$$ in $\Ho\bigl( \mV\Fun(\mB,\Psh_{\mV}(\mA))_{\mathrm{proj}} \bigr)$. Using the adjointness isomorphism $(\ast)$ of \Cref{lem:diagram-derived-adjunction-M}, the morphism $\gamma$ corresponds to a morphism $\epsilon_{\y}: FF^\vee\y_{\mB} \to \y_{\mB}$ in $\Ho \bigl(\mV\Fun(\mB,\Psh_{\mV}(\mB))_{\mathrm{proj}}\bigr).$ By \ref{lem:homotopical-Yoneda-comparison-M} we obtain a counit morphism $$\epsilon \colon   FF^\vee\to \id$$ in $\Pro_{\mV}(\mB,\mB)$.

We now construct the unit. To this end, consider the component $\epsilon_P$ at some cofibrant and fibrant $P\in\Psh_{\mV}(\mB)$. The pair $F \dashv G$ gives a bijection $$\Ho\Psh_{\mV}(\mB)(FF^\vee P,P) \cong \Ho\Psh_{\mV}(\mA)(F^\vee P,GP) \eqno{(\star)}$$ We define $$\Gamma_P: F^\vee P\to GP$$ to be the map corresponding to $\epsilon_P$ under the adjunction isomorphism $(\star)$. Since $\epsilon_{\y_{\mB}}$ was defined as the map corresponding to $\gamma: F^\vee\y_{\mB}\xrightarrow{\cong}G\y_{\mB}$ under the adjunction $(*)$ in \ref{lem:diagram-derived-adjunction-M}, it follows that $\Gamma_{\y_{\mB}(b)} = \gamma_b$. Now let $P\in\widetilde{\y(\mB)}.$ By definition (\ref{def:homotopy-representables-M}), $P$ is cofibrant-fibrant and isomorphic in $\Ho\Psh_{\mV}(\mB)$ to some $\y_{\mB}(b)$, hence, by naturality, $\Gamma_P$ is an isomorphism because  $\Gamma_{\y_{\mB}(b)} = \gamma_b$ is so. Now let $$X:\mA\to\widetilde{\y(\mB)}$$ the diagram such that $F \cong \overline{X}$. Evaluation on $X$ sends cofibrant-wise equivalences to objectwise weak equivalences, so $\epsilon$ gives $\epsilon_X:FF^\vee X\to X$. Its mate under \Cref{lem:diagram-derived-adjunction-M} is $\Gamma_X: F^\vee X \xrightarrow{\cong} GX$ in $\Ho\bigl(\mV\Fun(\mA,\Psh_{\mV}(\mA))_{\mathrm{proj}}\bigr).$ It is invertible since its components are the isomorphisms $\Gamma_{X(a)}$. Since $F \cong \overline X$, \ref{lem:homotopical-Yoneda-comparison-M} gives an isomorphism $\chi: F\y_{\mA} \xrightarrow{\cong} X$ in $\Ho\bigl(\mV\Fun(\mA,\Psh_{\mV}(\mB))_{\mathrm{proj}}\bigr)$. Let $$u_X: \y_{\mA}\to GX$$ be the map corresponding to $\chi$ under the adjunction isomorphism $(\ast)$ in \ref{lem:diagram-derived-adjunction-M}. We first define the restriction of the unit $\eta$ on $\y_{\mA}$ to be the composite $$\eta_{\y_{\mA}} = \left(  \y_{\mA} \xrightarrow{u_X} GX \xrightarrow{(\Gamma_X)^{-1}} F^\vee X \xrightarrow{F^\vee(\chi^{-1})} F^\vee F\y_{\mA}\right).$$ By \ref{lem:homotopical-Yoneda-comparison-M}, $\eta_{\y_{\mA}}$ uniquely determines a unit $2$-cell $$\eta: \id_{\mA}\Rightarrow F^\vee F$$ in $\Pro_{\mV}(\mA,\mA)$.

To prove the first triangular identity $(\epsilon F)\circ(F\eta)=\id_F$, it is enough, by \ref{lem:homotopical-Yoneda-comparison-M}, to restrict along $\y_{\mA}$. Let $T: F\y_{\mA} \to F\y_{\mA}$ be the restriction of $(\epsilon F)\circ(F\eta)$ along $\y_{\mA}$. Since $\chi: F\y_{\mA} \xrightarrow{\cong} X$ is an isomorphism in $\Pro_{\mV}(\mA,\mB)$, it is enough to prove $\chi\circ T=\chi.$ Writing $F$ on morphisms for $\mathbf L F_*$, the following diagram commutes: $$\begin{tikzcd}[column sep=4.5em,row sep=2.8em] F\y_{\mA} \ar[r,"F(\eta_{\y})"] \ar[dd,equal] & FF^\vee F\y_{\mA} \ar[r,"\epsilon_{F\y_{\mA}}"] \ar[d,"FF^\vee(\chi)"'] & F\y_{\mA} \ar[dd,"\chi"] \\ & FF^\vee X \ar[d,"F(\Gamma_X)"'] & \\ F\y_{\mA} \ar[r,"F(u_X)"'] & \mathbf L F_*(GX) \ar[r,"e_X"'] & X \end{tikzcd}$$ where $e_X:\mathbf L F_*(GX)\to X$ is the counit of $\mathbf L F_*\dashv\mathbf R G_*$, restricted at our $X:\mA\to\widetilde{\y(\mB)}$. Therefore $\chi\circ T=\chi$, proving the first triangle identity.

The other triangular identity is similarly obtained. Restricting along $\y_{\mB}$, we must show that $$F^\vee\y_{\mB} \xrightarrow{\eta_{F^\vee\y_{\mB}}} F^\vee FF^\vee\y_{\mB} \xrightarrow{F^\vee(\epsilon_{\y_{\mB}})} F^\vee\y_{\mB}$$ is the identity. Since $\gamma: F^\vee\y_{\mB} \xrightarrow{\cong} G\y_{\mB}$ is an isomorphism, it is enough to prove that $\gamma$ equals the composite $\gamma \circ F^\vee(\epsilon_{\y_{\mB}}) \circ \eta_{F^\vee\y_{\mB}} .$ The morphism $\gamma$ was defined as the transpose of $\epsilon_{\y_{\mB}}$ under the adjunction $F \dashv G$. Hence, the transpose of $\gamma \circ F^\vee(\epsilon_{\y_{\mB}}) \circ \eta_{F^\vee\y_{\mB}}$ is $\epsilon_{\y_{\mB}}\circ FF^\vee(\epsilon_{\y_{\mB}}) \circ F(\eta_{F^\vee\y_{\mB}})$, obtained as the composition of the maps in the diagram below $$\begin{tikzcd}[column sep=4.5em,row sep=3em] FF^\vee\y_{\mB} \ar[r,"F(\eta_{F^\vee\y_{\mB}})"] \ar[dr,equal] & FF^\vee FF^\vee\y_{\mB} \ar[r,"FF^\vee(\epsilon_{\y_{\mB}})"] \ar[d,"\epsilon_{FF^\vee\y_{\mB}}"'] & FF^\vee\y_{\mB} \ar[d,"\epsilon_{\y_{\mB}}"] \\ & FF^\vee\y_{\mB} \ar[r,"\epsilon_{\y_{\mB}}"'] & \y_{\mB}. \end{tikzcd}$$ The right square commutes by naturality of $\epsilon$, while the left triangle is the other triangular identity that we already proved. Thus, the transpose of  $\gamma \circ F^\vee(\epsilon_{\y_{\mB}}) \circ \eta_{F^\vee\y_{\mB}}$ is $\epsilon_{\y_{\mB}}$, and it therefore must be equal to $\gamma$.
\end{proof}
Keeping in mind the notation in the proof of \Cref{th:equipment-M}, $F \in \mV\Cat^h(\mA,\mB)$ is a vertical arrow, in particular a left Quillen functor $F \colon \operatorname{Psh}_{\mV}(\mA) \to \operatorname{Psh}_{\mV}(\mB)$ with $G$ is its ordinary right Quillen adjoint. There is a natural transformation
\begin{equation}\label{eq:reconstruction-comparison-M} g_P: F^{\vee}P=\overline{G\y_{\mB}}(P) = P\otimes_{\mB}Q(G\y_{\mB}) \to P\otimes_{\mB}G\y_{\mB} \to GP. \end{equation} The first map in the above composition is induced by a cofibrant replacement $q \colon Q(G\y_{\mB}) \to G\y_{\mB}$, and the second is given at $a\in\mA$ by $$\int^{b\in \mB} P(b)\otimes \Psh_{\mV}(\mB)(X(a),\y_{\mB}(b)) \to \Psh_{\mV}(\mB)(X(a),P).$$ Note that $g$ is not necessarily a $2$-cell of $\Pro_{\mV}$, as $G$ is right Quillen rather than left Quillen.

The following is a useful equivalent formulation of the existence of the equipment.
\begin{proposition}\label{prop:equipment-implies-reconstruction-M}
Let $F\in \mV\Cat^h(\mA,\mB)$. Then $F^{\vee}$ is right adjoint to $F$ in $\Pro_{\mV}$ if and only if for every cofibrant-fibrant $P\in\Psh_{\mV}(\mB)$, the map $$g_P:F^{\vee}(P)=\overline{G\y_{\mB}}(P)\to GP$$ is a weak equivalence in $\Psh_{\mV}(\mA)$.

\end{proposition}

\begin{proof}
We first prove the necessity. Consider the pseudo-functor from the $2$-category $\Pro_{\mV}$ to $\CAT$ which sends $\mC$ to $\Ho\Psh_{\mV}(\mC)$, sends a $1$-cell represented by a left Quillen functor $F$ to its left derived functor $\mathbf L F$, and a $2$-cell to the induced natural transformation. Since pseudo-functors preserve adjunctions, the adjointness relation $F\dashv_{\Pro_{\mV}}\overline{G\y_{\mB}}$ of \Cref{th:equipment-M} gives an adjunction $\mathbf L F\dashv_{\CAT}\mathbf L\overline{G\y_{\mB}}$. Independently, the Quillen adjunction $F:\Psh_{\mV}(\mA)\rightleftarrows\Psh_{\mV}(\mB):G$ gives an adjunction $\mathbf L F\dashv_{\CAT}\mathbf R G$. Let $\epsilon:F\overline{G\y_{\mB}}\to \id_{\Psh_{\mV}(\mB)}$ be the equipment counit, regarded as a morphism in $\Pro_{\mV}(\mB,\mB)$, and let $\epsilon^{\mathrm{ho}}: \mathbf L F\,\mathbf L\overline{G\y_{\mB}} \to\id_{\Ho\Psh_{\mV}(\mB)}$ be its image under the pseudofunctor above. Denote by $e^{\mathrm{ho}}: \mathbf L F\,\mathbf R G\to \id_{\Ho\Psh_{\mV}(\mB)}$ the counit at the level of homotopy categories induced by the one pertaining to the Quillen adjunction $e:FG\to\id_{\Psh_{\mV}(\mB)}$. By uniqueness of right adjoints, there is a unique natural isomorphism $$\vartheta: \mathbf L\overline{G\y_{\mB}} \xrightarrow{\cong} \mathbf R G$$ such that $e^{\mathrm{ho}}\circ\mathbf L F(\vartheta) =\epsilon^{\mathrm{ho}}$ as natural transformations of endofunctors of $\Ho\Psh_{\mV}(\mB)$. To conclude, it suffices to prove that, for every cofibrant-fibrant $P$, the component $\vartheta_{[P]}$ is the equivalence class of $g_P:\overline{G\y_{\mB}}P\to GP$ \eqref{eq:reconstruction-comparison-M}, since the latter would then have to be a weak equivalence as required. To this end, we first consider the restriction on $\y_{\mB}$. The composite $e\circ Fg: F\overline{G\y_{\mB}}\to\id_{\Psh_{\mV}(\mB)}$ ($e$ being the counit of $F \dashv G$) is an enriched natural transformation between left Quillen functors, and hence defines a morphism $[e\circ Fg]: F\overline{G\y_{\mB}}\to\id$ in $\Pro_{\mV}(\mB,\mB)$. By \Cref{lem:diagram-derived-adjunction-M} we have an isomorphism $$\Ho\bigl( \mV\Fun(\mB,\Psh_{\mV}(\mB))_{\mathrm{proj}} \bigr) \bigl(F\overline{G\y_{\mB}}\y_{\mB},\y_{\mB}\bigr) \cong \Ho\bigl( \mV\Fun(\mB,\Psh_{\mV}(\mA))_{\mathrm{proj}} \bigr) \bigl(\overline{G\y_{\mB}}\y_{\mB},G\y_{\mB}\bigr).$$ Under this isomorphism, the restriction of $[e\circ Fg]$ along $\y_{\mB}$ corresponds to $[g\y_{\mB}]$, which in turn identifies with the isomorphism $\gamma: \overline{G\y_{\mB}}\y_{\mB} \xrightarrow{\cong} G\y_{\mB}$ coming from the proof of \ref{lem:homotopical-Yoneda-comparison-M}; on the other hand, by definition, the restriction of $\epsilon$ along $\y_{\mB}$ is the element corresponding to $\gamma$ under the adjunction isomorphism above. Since restriction along $\y_{\mB}$ is faithful by \ref{lem:homotopical-Yoneda-comparison-M}, we obtain $[e\circ Fg]=\epsilon$ in $\Pro_{\mV}(\mB,\mB)$. Let now $P\in\Psh_{\mV}(\mB)$ be cofibrant-fibrant, so that $\overline{G\y_{\mB}}(P)$ is also cofibrant; since $G$ is right Quillen, $GP$ is fibrant. Hence the adjunction on homotopy categories is $$\Ho\Psh_{\mV}(\mB) \bigl(F\overline{G\y_{\mB}}(P),P\bigr) \cong \Ho\Psh_{\mV}(\mA) \bigl(\overline{G\y_{\mB}}(P),GP\bigr).$$ Under this isomorphism, the class $[e_P\circ F(g_P)]$ corresponds to $[g_P]$. Since we have proven that $[e\circ Fg]=\epsilon$, the identity $e^{\mathrm{ho}}\circ\mathbf L F(\vartheta) =\epsilon^{\mathrm{ho}}$, uniquely defining $\vartheta$, forces $\vartheta_{[P]}=[g_P]$ in $\Ho\Psh_{\mV}(\mA)$.

Conversely, we take as an assumption that for the given representable $\mV$-profunctor $F=\overline X:\mA\nrightarrow\mB$, with right Quillen adjoint $G$, the comparison morphism $g_P: F^{\vee}P=\overline{G\y_{\mB}}(P)\to GP$ is a weak equivalence for every cofibrant-fibrant $P\in\Psh_{\mV}(\mB)$, and we derive from this the adjointness relation $$F\dashv_{\Pro_{\mV}}F^{\vee}$$ Indeed, let $\mI$ be locally fibrant-cofibrant, and let $H\in\Pro_{\mV}(\mI,\mA)$ and $K\in\Pro_{\mV}(\mI,\mB)$.  Choose projectively cofibrant-fibrant replacements $Q(H\y_{\mI})$ and $Q(K\y_{\mI})$ of $H\y_{\mI}:\mI\to\Psh_{\mV}(\mA)$ and $ K\y_{\mI}:\mI\to\Psh_{\mV}(\mB)$. Then, recalling the definition of $\Pro_{\mV}$ (\ref{def:Pro-M-hom}) and \Cref{lem:homotopical-Yoneda-comparison-M}, $$\begin{aligned} \Pro_{\mV}(\mI,\mB)(FH,K) &\cong \Ho\bigl(\mV\Fun(\mI,\Psh_{\mV}(\mB))_{\mathrm{proj}}\bigr) (FQ(H\y_{\mI}),Q(K\y_{\mI}))                                                   \\ &\cong \Ho\bigl(\mV\Fun(\mI,\Psh_{\mV}(\mA))_{\mathrm{proj}}\bigr) (Q(H\y_{\mI}),GQ(K\y_{\mI}))                                                   \\ &\cong \Ho\bigl(\mV\Fun(\mI,\Psh_{\mV}(\mA))_{\mathrm{proj}}\bigr) \bigl(Q(H\y_{\mI}),\overline{G\y_{\mB}}Q(K\y_{\mI})\bigr)        \\ & \cong  \Pro_{\mV}(\mI,\mA)\bigl(H,\overline{G\y_{\mB}}K\bigr). \end{aligned}$$ The second isomorphism is \Cref{lem:diagram-derived-adjunction-M}. The third isomorphism is where we use our assumption that $g_P:\overline{G\y_{\mB}}P\to GP$ is a weak equivalence provided that $P=Q(K\y_{\mI})(i)$ is cofibrant-fibrant, which is true by \Cref{cor:projectively-cofibrant-objectwise-cofibrant-M}.
\end{proof}
\subsubsection{Closedness}\label{prop:closed-equipment-M}
The equipment $\mV\Cat^h\to\Pro_{\mV}$ is \emph{closed} in the following $\Ho\mV$-enriched sense. Given a cospan of horizontal arrows $$\begin{tikzcd}[row sep=large, column sep=large] \mA \arrow[dr, "Z"', "\shortmid"{marking}] && \mB \arrow[dl, "H", "\shortmid"{marking}] \\ & \mC & \end{tikzcd}$$ there is a horizontal arrow $$[H,Z] : \mA\nrightarrow\mB$$ and a $2$-cell $\mathrm{ev}_H \colon H[H,Z] \to Z$ such that for any $X$ as in the diagram $$
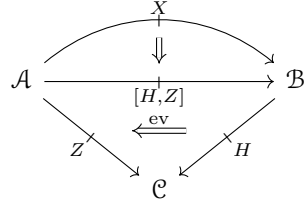
\begin{tikzcd}[row sep=large, column sep=large] \mA \arrow[rr, bend left=40, "X", "\shortmid"{marking}, ""{name=X, below}] \arrow[rr, "{[H,Z]}"', "\shortmid"{marking}, ""{name=K, above}] \arrow[dr, "Z"', "\shortmid"{marking}, ""{name=Z, above}] && \mB \arrow[dl, "H", "\shortmid"{marking}, ""{name=H, above}] \arrow[Rightarrow, from=X, to=K, shorten <=2pt, shorten >=2pt] \\ & \mC \arrow[Rightarrow, from=H, to=Z, shorten <=13pt, shorten >=13pt, "\mathrm{ev}"'] & \end{tikzcd}$$ composition with $\mathrm{ev}_H$ induces natural isomorphisms in $\Ho\mV$, $$\underline{\Pro_{\mV}}(\mA,\mC)(HX,Z)\cong \underline{\Pro_{\mV}}(\mA,\mB) \bigl(X,[H,Z]\bigr),$$ and similarly for every span of horizontal arrows $$
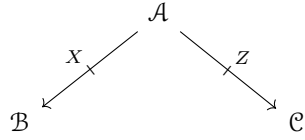
\begin{tikzcd}[row sep=large, column sep=large] & \mA \arrow[dl, "X"', "\shortmid"{marking}] \arrow[dr, "Z", "\shortmid"{marking}] & \\ \mB && \mC \end{tikzcd}$$ there is $$\{X,Z\} : \mB\nrightarrow\mC$$ such that $$\underline{\Pro_{\mV}}(\mA,\mC)(HX,Z) \cong \underline{\Pro_{\mV}}(\mB,\mC) \bigl(H,\{X,Z\}\bigr)$$ naturally as objects of $\Ho \mV$.

At the level of $\pi_0$, we are saying that the equipment's composition admits a right adjoint separately in each variable \cite{Law71}.
\begin{proof}
Let $X:\mA\nrightarrow\mB$, $H:\mB\nrightarrow\mC$, and $Z:\mA\nrightarrow\mC$ be horizontal arrows as above. By definition these are given as left Quillen functors between the respective categories of presheaves. We choose a projective cofibrant replacement $\widetilde X\xrightarrow{\sim}X\y_{\mA}.$ There are enriched Quillen adjunctions $$\widetilde X\otimes_{\mB}-: \mV\Fun(\mB,\Psh_{\mV}(\mC))_{\mathrm{proj}} \rightleftarrows \mV\Fun(\mA,\Psh_{\mV}(\mC))_{\mathrm{proj}} : G^r(\widetilde X,-)$$ and $$-\otimes_{\mB}(H\y_{\mB}): \mV\Fun(\mA,\Psh_{\mV}(\mB))_{\mathrm{proj}} \rightleftarrows \mV\Fun(\mA,\Psh_{\mV}(\mC))_{\mathrm{proj}} : G^{\ell}(H,-),$$ where $G^{r}(\widetilde X,P):\mB\to\Psh_{\mV}(\mC)$ and $G^{\ell}(H,P):\mA\to\Psh_{\mV}(\mB)$ are given as $$G^{r}(\widetilde X,P)(b)(c) = \underline{\mV\Fun(\mA,\mV)} \bigl(\widetilde X(-)(b),P(-)(c)\bigr)$$ and $$G^{\ell}(H,P)(a)(b) = \underline{\Psh_{\mV}(\mC)}\bigl(H\y_{\mB}(b),P(a)\bigr).$$ We now fix a projectively fibrant replacement $Z\y_{\mA}\xrightarrow{\sim}RZ$ in $\mV\Fun(\mA,\Psh_{\mV}(\mC))_{\mathrm{proj}}$ and we introduce $$\{X,Z\} \coloneqq \overline{G^r(\widetilde X,RZ)} : \mB\nrightarrow\mC$$ and $$[H,Z] \coloneqq \overline{G^\ell(H,RZ)} : \mA\nrightarrow\mB.$$ Passing to the induced $\Ho\mV$-enriched adjunctions \cite[16.8]{Shu06} and using \Cref{lem:homotopical-Yoneda-comparison-M}, we obtain natural isomorphisms in $\Ho\mV$ $$\underline{\Pro_{\mV}}(\mA,\mC)(HX,Z) \cong \underline{\Pro_{\mV}}(\mB,\mC) \bigl(H,\{X,Z\}\bigr)$$ and $$\underline{\Pro_{\mV}}(\mA,\mC)(HX,Z) \cong \underline{\Pro_{\mV}}(\mA,\mB) \bigl(X,[H,Z]\bigr).$$ 
\end{proof}

\section{Weak category theory of enriched categories}\label{sec:weakcategorytheory}

Let $f:\mC\to\mB$ be an ordinary $\mV$-functor. We want to see it as a vertical arrow in the equipment. To this end, we consider the composition $\y_{\mB}f$: it is cofibrant-valued, hence its enriched left Kan extension is left Quillen and we may take $\overline{\y_{\mB}f}=\Lan_{\y_{\mC}}\y_{\mB}f$, where $\overline{(-)}$ is the functor defined in \Cref{lem:homotopical-Yoneda-comparison-M}.

\begin{notation}\label{not:verticalstrictfunctor}
We denote $$f_!=\overline{\y_{\mB}f}$$ the vertical arrow associated to a strict $\mV$-functor $f$.
\end{notation}

Its right Quillen adjoint is restriction along $f$, denoted $\res_f: \Psh_{\mV}(\mB)\to\Psh_{\mV}(\mC)$. The composition $\res_f\y_{\mB}:\mB\to\Psh_{\mV}(\mC)$ is the functor $b\to\mB(f-,b)$.

\begin{notation}\label{not:conjointofvertical}
The conjoint of the vertical arrow $f_!$ is defined to be $(f_!)^\vee=\overline{\res_f\y_{\mB}}$. We abbreviate the notation as $$f^!:=(f_!)^\vee.$$
\end{notation}

\begin{remark}
Not all vertical arrows in $\mV\Cat^h$ arise this way from ordinary $\mV$-functors. A general vertical arrow is interpreted as a \emph{weak $\mV$-functor}, while those of type $f_!$ are interpreted as \emph{strict $\mV$-functors}. 
\begin{notation}\label{conv:companion} In any equipment, a vertical arrow $F$ can be seen as an horizontal one via its companion $F_*$. In the equipment constructed in \Cref{th:equipment-M}, the functor $(-)_*$ is the identity (it is an inclusion). Therefore we do not distinguish notationally between $F$ and its companion, and between vertical and horizontal identities (both are $\id_{\mC}$).

\end{notation}

\end{remark}

\subsection{Homotopy colimits are formal colimits}\label{sec:colim}

\subsubsection{Weighted colimits in an equipment}
Recall that, in ordinarily enriched category theory (for simplicity, let us even assume enriched over $\Set$), a weight is taken to be a functor $W:\mC^{\op}\to\Set$; it is useful to consider at once a family of weights $W_a$ parametrized by a small category $\mA$, so that a weight in this broader sense is a functor $W:\mC^{\op}\times\mA\to\Set$ (in the language adopted here, it is a horizontal arrow $W \colon \mA \nrightarrow \mC$ in the equipment of ordinary categories); the weighted colimit of a chosen diagram $D:\mC\to\mE$, when it exists, is therefore a functor $a \mapsto \colim^{W_a}D$, and, given a functor $L:\mA\to\mE$, a $W$-weighted cocone on $D$ with vertex $L$ is a natural transformation $\{ \lambda_{c,a} \colon W(c,a)\to\mE(Dc,La)\}_{c,a}$.

With these considerations in mind we approach the following definitions.
\begin{definition}
Let $D \colon \mC \to \mE$ and $L \colon \mA \to \mE$ be vertical arrows, and let $W \colon \mA \nrightarrow \mC$ be a horizontal arrow. A \emph{cocone on $D$ weighted by $W$ with vertex $L$} is a square $$\begin{tikzcd}[column sep=3.5em,row sep=3em] \mA \ar[r,"W"] \ar[d,"L"'] & \mC \ar[d,equal] \\ \mE \ar[r,"D^{\vee}"'] & \mC \arrow[phantom,from=1-1,to=2-2,"\Downarrow\lambda" description] \end{tikzcd}$$ We denote $$\operatorname{Cocones}^W(D,L) \coloneqq \underline{\Pro_{\mV}}(\mA,\mC)\bigl(W,D^{\vee}L\bigr) \in \operatorname{Ho}\mV$$ \emph{the homotopy type of $W$-weighted cocones on $D$ with vertex $L$}. 
\end{definition} 
We recall that $\Pro_{\mV}(\mA,\mC)(X,Y)=\operatorname{\pi_0}\underline{\Pro_{\mV}}(\mA,\mC)(X,Y)$, where $\operatorname{\pi_0}\coloneqq \bigl(\operatorname{Ho}\mV\bigr) \bigl(\tu,-\bigr)$.  There is an isomorphism in $\Ho \mV$ $$(-)^{\#} \colon \operatorname{Cocones}^W(D,L) \xrightarrow{\cong} \underline{\Pro_{\mV}}(\mE,\mA)\bigl(L^{\vee},\left[W,D^{\vee}\right]\bigr)$$ owing to the adjointness $L \dashv L^{\vee}$ and the closedness of the equipment. At the level of $\pi_0$, the bijection is obtained precomposing a cocone $\lambda \colon W \to D^{\vee}L$ with $L^\vee$, and then composing with the counit map $LL^{\vee} \to \id_{\mE}$, so as to obtain a $2$-cell $WL^{\vee} \Rightarrow D^{\vee}$, which by closedness corresponds to a morphism that we denote $$\lambda^{\#}\colon L^{\vee} \to \left[W,D^{\vee}\right].$$ 

Next, for a $\mV$-category $\mJ$, a horizontal arrow $H:\mJ\nrightarrow\mA$, and a vertical arrow $V:\mJ\to\mE$ together with a $2$-cell $\gamma \colon LH \rightarrow V$, there is a cocone on $D$ with vertex $V$ weighted by $WH$, obtained as the pasted diagram $$\begin{tikzcd}[column sep=3.5em,row sep=3em] \mJ \ar[r,"H"] \ar[d,"V"'] & \mA \ar[r,"W"] \ar[d,"L"'] & \mC \ar[d,equal] \\ \mE \ar[r,equal] & \mE \ar[r,"D^{\vee}"'] & \mC \arrow[phantom,from=1-1,to=2-2, "\Downarrow\gamma" description] \arrow[phantom,from=1-2,to=2-3, "\Downarrow\lambda" description] \end{tikzcd}$$ This way composition with $\lambda$ induces a morphism in $\operatorname{Ho}\mV$ $$\lambda_{H,V}^{\circ}\colon \underline{\Pro_{\mV}}(\mJ,\mE)(LH,V) \to \operatorname{Cocones}^{WH}(D,V).$$ Moreover, taking $\mJ=\mA$, and $H=\id_{\mA}$, we see that $$\left[W,D^\vee V\right] \colon \mA \nrightarrow \mA$$ is a horizontal arrow whose global sections are precisely $W$-cocones on $D$ with vertex $V$, and $\lambda_{H,V}^{\circ}$ above is the morphism $\lambda_{H,V}^{\circ} = \underline{\Pro_{\mV}}(\mJ,\mA)(H,\rho_V)$ induced by the map $$\rho_V \colon \left[L,V\right] \to \left[W,D^\vee V\right]$$ which, under closedness, is the mate of $WL^{\vee}V \xrightarrow{\lambda L^{\vee}V} D^{\vee}LL^{\vee}V \xrightarrow{D^{\vee}\epsilon V} D^{\vee}V$, using that $L^{\vee}V \cong \left[L,V\right]$, the latter isomorphism owing to $L \dashv L^{\vee}$ (since $L$ is vertical) and the closedness property of the equipment; for the same reason, since $V$ is vertical, one has that $\left[W,D^{\vee}V\right] \cong \left[W,D^{\vee}\right]V$, so that $\rho_V$ identifies with $\lambda^{\#}V$. Finally, Yoneda lemma ensures that bijectivity of $\pi_0\lambda_{H,V}^{\circ}$ for every $H,V$ implies invertibility of $\rho_V$.

We have proven the following.
\begin{proposition}\label{prop:weightedcolimdef}

Let $\lambda\in\pi_0\operatorname{Cocones}^W(D,L)$. The following are equivalent:
\begin{itemize}
\item[(i)] the morphism ${\lambda}^{\#}: L^\vee\to\left[W,D^\vee\right]$ is invertible in $\Pro_{\mV}(\mE,\mA)$;
\item[(ii)] for every $\mJ$ and every vertical arrow $V$, the map $\rho_V $ is invertible in $\Pro_{\mV}(\mJ,\mA)$;
\item[(iii)] for every $\mJ$, $H$, and $V$ as above, the morphism $\lambda_{H,V}^\circ$ is an isomorphism in $\Ho\mV$;
\item[(iv)] for every $\mJ$, $H$, and $V$ as above, $\pi_0\lambda_{H,V}^\circ$ is bijective;
\end{itemize}
\end{proposition}

\begin{definition}\label{def:weightedcolimdef}
\emph{A cocone $\lambda$ exhibits $L$ as the $W$-weighted colimit of $D$} if the equivalent conditions of \Cref{prop:weightedcolimdef} are satisfied. We write $L=\colim^W D$.
\end{definition}

Weighted colimits in an equipment satisfy Fubini's theorem. \cite[Proposition~6]{Woo82}.

\begin{proposition}[Fubini's theorem]
\label{prop:associativity-weighted-colimits-M}
Let $$W:\mA\nrightarrow\mC, \qquad F:\mC\to\mD, \qquad E:\mD\to\mE$$ be respectively a horizontal arrow and two vertical arrows. $$\colim^W(EF)\cong\colim^{FW}E$$ whenever either side exists.
\end{proposition}

\begin{proof}
Using the adjunction $F\dashv F^\vee$ and (i) in \ref{prop:weightedcolimdef}, a cocone $\lambda: W\to(EF)^\vee L \cong F^\vee E^\vee L$ exhibits $L$ as the $W$-weighted colimit of $EF$ if and only if its mate $\theta \colon FW \xrightarrow{F\lambda} FF^\vee E^\vee L \xrightarrow{\epsilon_F E^\vee L} E^\vee L$ exhibits $L$ as the $FW$-weighted colimit of $E$. Indeed, the map $\theta^{\#}\colon L^{\vee} \to \left[FW,E^{\vee}\right]$ can be defined in terms of $\lambda^{\#}$ as the composition $L^{\vee} \xrightarrow{\lambda^{\#}} \left[W,(EF)^{\vee}\right] \cong \left[W,F^{\vee}E^{\vee}\right]\cong \left[FW,E^{\vee}\right]$.
\end{proof}
Fubini's theorem relates composition in the equipment to weighted colimits as follows. 
\begin{corollary}\label{cor:formal-weighted-colimit-M}
Let $W:\mA\nrightarrow\mC,$ and $D:\mC\to\mE,$ be an horizontal and a vertical arrow respectively. We have $$\colim^W D\cong\colim^{DW}\id_{\mE}$$ whenever either side exists.
\end{corollary}

\subsubsection{Pointwise Kan extensions}\label{def:pointwise-left-Kan-extension-M}
Pointwise Kan extensions in $2$-categories were introduced  by Street \cite{Str74}. In the language of equipments, the definition below is due to Wood \cite{Woo82}.

Let $J:\mC\to\mP$, $D:\mC\to\mE$, and $L:\mP\to\mE$ be vertical arrows, and let $\eta:D\to LJ$ be a $2$-cell. Under the companion-conjoint adjunction for $J$, the cell $\eta$ corresponds to a cocone $\widehat\eta:J^\vee\to D^\vee L$: $$\begin{tikzcd}[column sep=3.5em,row sep=3em] \mC \ar[r,"D"] \ar[d,"J"'] & \mE \ar[d,equal] \\ \mP \ar[r,"L"'] & \mE \arrow[phantom,from=1-1,to=2-2,"\Downarrow\eta" description] \end{tikzcd} \quad \begin{tikzcd}[column sep=3.5em,row sep=3em] \mP \ar[r,"J^\vee"] \ar[d,"L"'] & \mC \ar[d,equal] \\ \mE \ar[r,"D^\vee"'] & \mC \arrow[phantom,from=1-1,to=2-2,"\Downarrow\widehat\eta" description] \end{tikzcd}$$ We say that $\eta$ exhibits $L$ as the \emph{pointwise left Kan extension} of $D$ along $J$ if $\widehat\eta$ exhibits $L$ as the $J^\vee$-weighted colimit of $D$. We then write $L=\Lan_JD$. In formulas $$\Lan_JD \coloneqq \colim^{J^\vee}D$$

\subsubsection{Homotopy colimits as formal colimits}\label{sec:hocolim}

Let $W:\mA\nrightarrow\mC$ be a horizontal arrow, which by definition is a left Quillen functor $\operatorname{Psh}_{\mV}(\mA) \to \operatorname{Psh}_{\mV}(\mC)$, and let $d:\mC\to\mE$ and $\ell:\mA\to\mE$ be strict $\mV$-functors, and $$\lambda:W\to d^!\ell_!$$ a $W$-weighted cocone on $d_!$ with vertex $\ell_!$. Pick $\widetilde{W} \colon \mA \to \Psh_{\mV}(\mC)$ a $\mV$-functor cofibrant in $\mV\Fun(\mA,\Psh_{\mV}(\mC))_{\mathrm{proj}}$ with $W = \Lan_{\y}\widetilde{W}$. After restriction along $\y_{\mA}$, the cocone determines a morphism $\widetilde{W}\to\mE(d-,\ell-)$ in $\Ho\bigl( \mV\Fun (\mA,\Psh_{\mV}(\mC))_{\mathrm{proj}} \bigr)$, where $\mE(d-,\ell-)(a)=\mE(d-,\ell a)$. Since $\widetilde W$ is projectively cofibrant and $\mE(d-,\ell-)$ is always projectively fibrant, we may choose a representative enriched natural transformation $\widetilde W\to\mE(d-,\ell-)$, which induces, for $a\in\mA$ and $e\in\mE$, a morphism
\begin{equation}\label{eq:weighted-colimit-comparison-M} \rho_{a,e}: \mE(\ell a,e) \to \Psh_{\mV}(\mC) \bigl(\widetilde W(a),\mE(d-,e)\bigr) \end{equation}

\begin{proposition}\label{cor:weighted-colimit-strict-M}
The cocone $\lambda$ exhibits $\ell_!$ as the $W$-weighted colimit of $d_!$ in the equipment (according to \ref{def:weightedcolimdef}) if and only if the morphism $\rho_{a,e}$ of \eqref{eq:weighted-colimit-comparison-M} is a weak equivalence in $\mV$ for every $a\in\mA$ and $e\in\mE$.
\end{proposition}

\begin{proof}
The conjoint $\ell^!$ is induced by left Kan extension of $\mE \ni e\mapsto\mE(\ell -,e) \in \operatorname{Psh}_{\mV}(\mA)$, while $[W,d^!]_{}$ by construction is obtained extending the functor which sends $e$ to the presheaf on $\mA$ whose value at $a$ is the object $\underline{\operatorname{Psh}_{\mV}(\mC)}(\widetilde W_a,\mE(d-,e)) \in \mV$. Hence, the restriction on $\y_{\mE}$ of ${\lambda}^{\#}: \ell^!\to[W,d^!]_{}$ identifies in the component $(a,e)$ with $\rho_{a,e}$. We conclude by \Cref{lem:homotopical-Yoneda-comparison-M}.
\end{proof}
Homotopy colimits can be reduced to strict weighted colimits, provided the latter exist, as follows.  
\begin{corollary}\label{cor:cofibrant-weight-computes-colimit-M}
Suppose that the $\widetilde W$-weighted colimit (in the sense of ordinary enriched category theory) of $d$ exists, and denote it by $\ell = \widetilde W \otimes_{\mC}d \colon \mA \to \mE.$ Then $$\ell_!=\colim^W d_!$$ in the sense of \ref{def:weightedcolimdef}.
\end{corollary}
\begin{proof}
    In this case, the map $\rho_{a,e}$ is an isomorphism.
\end{proof}
Colimits of general vertical arrows (weak $\mV$-functors) can be reduced to colimits of strict $\mV$-functors by Fubini's theorem.
\begin{corollary}\label{cor:general-weighted-colimit-computation-M}
Let $D:\mC\to\mE$ be a vertical arrow, represented by some $\widetilde D:\mC\to\widetilde{\y(\mE)}$. Let $W$ a horizontal arrow and assume $\widetilde{W} \colon \mA \to \operatorname{Psh}_{\mV}(\mC)$ projectively cofibrant in $\mV\Fun\bigl(\mA,\operatorname{Psh}_{\mV}(\mC)\bigr)$. Suppose that the strict $(\widetilde W\otimes_{\mC}\widetilde D)$-weighted colimit of $\id_{\mE}$ exists, and denote it $$\ell = \bigl( \widetilde W\otimes_{\mC}\widetilde D \bigr) \otimes_{\mE}\id_{\mE}: \mA\to\mE.$$ Then $$\ell_!=\colim^W D$$
\end{corollary}

\begin{proof}
The composite $DW$ is computed Kan extending the functor $\widetilde W\otimes_{\mC}\widetilde D: \mA\to\Psh_{\mV}(\mE)$. Since $\widetilde W$ is projectively cofibrant, $\widetilde W\otimes_{\mC}\widetilde D$ is too by \ref{lem:diagram-derived-adjunction-M}. Fubini's theorem \ref{prop:associativity-weighted-colimits-M} then ensures that $\colim^{DW} \id_{\mE} \cong \colim^WD$, and we conclude by \Cref{cor:cofibrant-weight-computes-colimit-M}.
\end{proof}
\subsubsection{A formula for homotopy colimits}\label{sec:computecolim}
Let us explicitate the formula in \ref{cor:cofibrant-weight-computes-colimit-M} using the cofibrant replacement in \Cref{ap:A}. Consider the composition $W\y_{\mA}:\mA\to\Psh_{\mV}(\mC)$. Since $W$ is left Quillen, each $W\y_{\mA}(a)$ is projectively cofibrant. Choose a projective cofibrant replacement $QW\y_{\mA}\xrightarrow{\sim}W\y_{\mA}$ in $\mV\Fun (\mA,\Psh_{\mV}(\mC))_{\mathrm{proj}}$. By \Cref{lem:homotopical-Yoneda-comparison-M}, $W \cong \overline {W\y_{\mA}}=\Lan_{\y_{\mA}}QW\y_{\mA}$. In other words, we may take $\widetilde W=QW\y_{\mA}$ in \ref{cor:cofibrant-weight-computes-colimit-M}, \ref{cor:cofibrant-weight-computes-colimit-M} and \ref{cor:general-weighted-colimit-computation-M}. The hypotheses in \Cref{prop:explicit-projective-replacement-M} are met, so that, whenever the indicated strict weighted colimit exists, and assuming the base of enrichment $\mV$ is a simplicial model category (so that $\Delta^{\bullet}$ is a cosimplicial resolution of the monoidal unit), $$(\colim^Wd)(a)=\int^{c\in\mC}\int^{[n]\in\Delta_{\mathrm{inj}}} \Delta^n\otimes B_nW\y_{\mA}(a)(c) \otimes d(c)$$ while for a general $\mV$ $$(\colim^Wd)(a) = \int^{c\in\mC} |\!|\!| B_\bullet W\y_{\mA}(a)(c)|\!|\!|\otimes d(c)$$ where $|\!|\!|-|\!|\!|$ is the fat realization defined in \ref{not:realization}, and where $$B_nW\y_{\mA}(a)(c) =$$$$\coprod_{\substack{a_0,\ldots,a_n\\c_0,\ldots,c_n}} \mA(a_n,a)\otimes\mC(c,c_n) \otimes\mA(a_{n-1},a_n)\otimes\mC(c_n,c_{n-1})\otimes \cdots\otimes \mA(a_0,a_1)\otimes\mC(c_1,c_0) \otimes W\y_{\mA}(a_0)(c_0) \label{eq:weighted-colimit-explicit-weight-M}$$
\subsection{An enriched Quillen's Theorem A}
\label{subsec:enriched-Quillen-A-M}
We will, at first, consider the case where the monoidal unit $\tu_{\mV}$ is terminal in $\mV$. We denote $\mathbf 1$ the terminal $\mV$-category and $t_{\mC} \colon \mC \to \mathbf1$ the unique $\mV$-functor with its associated vertical arrow $(t_{\mC})_!$, whose conjoint we denote $$t_{\mC}^! := \bigl((t_{\mC})_!\bigr)^\vee \colon \mathbf 1 \nrightarrow \mC$$ The right adjoint of the left Quillen functor representing $(t_{\mC})_!$ sends $E\in\mV$ to the constant presheaf on $\mC$ with value $E$. It follows that $t_{\mC}^!$ is induced by a projectively cofibrant replacement $Q_{\mC}(\cst_{\tu_{\mV}}) \to \cst_{\tu_{\mV}}$ of the constant-unit presheaf on $\mC$.

For a vertical arrow $F:\mA\to\mC$, there is a canonical invertible $2$-cell $$ \alpha_F:(t_{\mC})_!F\to (t_{\mA})_! $$ Taking the mate we obtain $$\beta_F \colon Ft_{\mA}^! \to t_{\mC}^!$$

\subsubsection{The $\mV$-realization} \label{sec:realization}Given a horizontal arrow $H \colon \mC \nrightarrow \mathbf{1}$, we define its $\mV$\emph{-realization} $\lVert H \rVert_{\mV}$ to be the homotopy type obtained as the composition $Ht_{\mC}^!$. Therefore we have $$\lVert H \rVert_{\mV} \in \Pro_{\mV}\bigl(\mathbf{1},\mathbf{1}\big) \simeq \operatorname{Ho}\mV.$$ 
For a general $H \colon \mA \nrightarrow \mB$ it is also convenient to write $\lVert H \rVert_{\mV} = H t_{\mA}^! \in \Ho \Psh_{\mV}\bigl(\mB\bigr)$.

\begin{definition}\label{def:classifying-functor-M}
We define the \emph{classifying type of a $\mV$-category} as the functor $$\mathbb B_{\mV} \colon \mV\Cat^h \to \Ho\mV$$ whose value at $\mC$ is computed by $\mV$-realizing the terminal arrow $(t_{\mC})_{!} \colon \mC \to \mathbf{1}$, $$\mathbb B_{\mV}\mC \coloneqq \lVert (t_{\mC})_{!} \rVert_{\mV}$$ and, for a vertical arrow $F:\mA\to\mC$, $$\mathbb B_{\mV}F = \bigl((t_{\mC})_!\beta_F\bigr) \circ \bigl(\alpha_Ft_{\mA}^!\bigr)^{-1}: \mathbb B_{\mV}\mA \to \mathbb B_{\mV}\mC.$$
\end{definition}

It can be computed as $$\mathbb B_{\mV}\mC \cong Q_{\mC}(\cst_{\tu_{\mV}}) \otimes_{\mC} \cst_{\tu_{\mV}}$$ in $\Ho\mV$, where $Q_{\mC}$ denotes a projectively cofibrant replacement in $\operatorname{Psh}_{\mV}\bigl(\mC\bigr)$.
\begin{definition}\label{def:commamodule}
    Let $c:\mathbf{1} \to \mC$ an object of $\mC$, and $c^! \colon \mC \to \mathbf{1}$ its conjoint. For a vertical arrow $F \colon \mA \to \mC$, we call $c^!F \colon \mA \nrightarrow\mathbf{1}$ its \emph{comma module}, and we denote it $$\mC(c,F-) \coloneqq c^!F $$
\end{definition}
Let $\widetilde{F} \colon \mA \to \widetilde{\y(\mC)} \subset \operatorname{Psh}_{\mV}(\mC)$ be a $\mV$-functor inducing the vertical arrow $F$ as $F=\operatorname{Lan}_{\y_{\mA}}\widetilde{F}$, we have $$\lVert\mC(c,F-)\rVert_{\mV} \cong Q_{\mA}(\cst_{\tu_{\mV}}) \otimes_{\mA} \widetilde F(-)(c)$$ Following \Cref{ap:A}, we may give formulas by computing the cofibrant replacement $Q$. Consider the simplicial objects in $\mV$ $$n \mapsto B_n\mA = \coprod_{a_0,\ldots,a_n} \mA(a_{n-1},a_n)\otimes\cdots\otimes \mA(a_0,a_1)$$ and similarly $$B_n(\widetilde F(-)(c)) = \coprod_{a_0,\ldots,a_n} \mA(a_{n-1},a_n)\otimes\cdots\otimes \mA(a_0,a_1)\otimes \widetilde F(a_0)(c),$$ where, if $F=f_!$ is induced by an ordinary $\mV$-functor $f:\mA\to\mC$, we may take $\widetilde F(a)(c)=\mC(c,fa)$.
\begin{proposition}\label{prop:classifying-object-resolution-M}
Let $| \! | \! | - |\!|\!|$ be the fat realization defined in \ref{not:realization}. There are natural isomorphisms in $\Ho\mV$: $$\mathbb B_{\mV}\mA \cong | \! | \! |B_\bullet\mA| \! | \! |, \qquad \lVert\mC(c,F-)\rVert_{\mV} \cong | \! | \! |B_\bullet(\widetilde F(-)(c))| \! | \! |.$$
\end{proposition}

\begin{proof}
Given that $(-)\otimes_{\mA}\cst_{\tu_{\mV}}$ preserves colimits, we reduce to \Cref{prop:explicit-projective-replacement-M}.
\end{proof}

\subsubsection{Enriched Theorem A}
\begin{theorem}
\label{th:enriched-Quillen-A-M}
Let $$F:\mA\to\mC$$ be a vertical arrow, and assume $\lVert\mC(c,F-)\rVert_{\mV}$ is contractible in $\mV$ for every $c \in \mC$. Then for every vertical arrow $$D:\mC\to\mE$$ if either $\colim(DF)$ or $\colim D$ exists the other does too, and the comparison map $$\colim(DF)\to \colim D$$ is invertible.

\end{theorem}
\begin{proof}
 By Fubini's theorem \ref{prop:associativity-weighted-colimits-M} we have $\colim(DF) =\colim^{t_{\mA}^!}(DF) \cong\colim^{Ft_{\mA}^!}D $. Our assumption on the contractibility of $\lVert\mC(c,F-)\rVert_{\mV}$ for every $c$ says that the map $\lVert\mC(-,F-)\rVert_{\mV}=Ft_{\mA}^! \xrightarrow{\beta_F} t_{\mC}^{!}$ is invertible, therefore $\colim^{Ft_{\mA}^!}D\cong\colim^{t_{\mC}^!}D =\colim D$.
\end{proof}
\begin{corollary} Under the assumptions of \Cref{th:enriched-Quillen-A-M}, the map $$\mathbb B_{\mV}F: \mathbb B_{\mV}\mA \to \mathbb B_{\mV}\mC$$ is an isomorphism in $\Ho\mV$.
\end{corollary}
\begin{proof}
By \Cref{def:classifying-functor-M} $\mathbb B_{\mV}F = \bigl((t_{\mC})_!\beta_F\bigr) \circ \bigl(\alpha_Ft_{\mA}^!\bigr)^{-1}$ is invertible since both $\alpha_F$ and $\beta_F$ are.
\end{proof}

\begin{definition}
    A vertical arrow $F \colon \mA \to \mC$ is \emph{final} if the condition of \ref{th:enriched-Quillen-A-M} holds.
\end{definition}

\subsection{Weighted finality}

Again by Fubini's theorem, we can formulate the following more general version. 

Given weights, $P \colon \mJ \nrightarrow \mA $ and $W \colon \mJ \nrightarrow \mC$ and a vertical arrow $F \colon \mA \to \mC$ together with a $2$-cell $\alpha \colon P \to F^{\vee}W$, if the mate under the duality $\beta \colon FP \to W$ is invertible, then for every vertical arrow $D \colon \mC \to \mE$, the comparison map induced by $\alpha$ $$\operatorname{colim}^{P}(DF)
\xrightarrow{}
\operatorname{colim}^{W}D$$ is invertible whenever either side exists.
If $F=f_!$ is induced by a strict $\mV$-functor $f:\mA\to\mC$, and assuming $P= \Lan_{\y_{\mJ}}\widetilde P$ for a projectively cofibrant $\widetilde P$, then a sufficent condition is that the morphism
induced by $\beta$ $$\widetilde P(j)\otimes_{\mA}\mC(c,f-)
\longrightarrow
\bigl(W\y_{\mJ}(j)\bigr)(c)$$
is an isomorphism in $\Ho\mV$ for every $j\in\mJ$ and $c\in\mC$.

Changing the base of enrichment, or varying the model structure on it, the notion of finality changes accordingly. Below are examples.
\subsection{Some computations of classifying objects}\label{sec:comutationsclassifying}

\subsubsection{Homotopy finality}\label{homotopyfinality}

Let $\mV=\sSet_{\mathrm{KQ}}$. The hypotheses of \ref{rem:locwellp} are met, so that for a locally Kan simplicial category $\mA$, the classifying space $\mathbb{B}_{\mathrm{KQ}}$ defined in $\ref{def:classifying-functor-M}$ has the homotopy type of the realization (computed by taking the diagonal) of the bisimplicial set produced by an application of the nerve $N$ to $\mA_n$ for each $n$; we write $N_{\bullet}(\mA_{\bullet})$ for it. To give a formula, its $n$-simplexes are $$\bigl(\operatorname{diag}N_{\bullet}(\mA_{\bullet})\bigr)_n=(N(\mA_n))_n \cong \coprod_{a_0,\ldots,a_n} \mA(a_0,a_1)_n \times\cdots\times \mA(a_{n-1},a_n)_n .$$ Invoking \cite[1.24 and 2.7]{Joy07} together with \cite[5.6 and 5.7]{JT07} (cf. \cite[4.3]{Ara25}), we get an isomorphism $\operatorname{diag}N_{\bullet}(\mA_{\bullet}) \cong N^{\mathrm{hc}}(\mA)$ naturally in $\Ho\sSet_{\mathrm{Joyal}}$, where $N^{\mathrm{hc}}(\mA)$ is the homotopy coherent nerve, and therefore $\mathbb B_{\mathrm{KQ}}\mA \cong N^{\mathrm{hc}}(\mA)$ naturally in $\Ho\sSet_{\mathrm{KQ}}$.

It follows that $\mathbb{B}_{\mathrm{KQ}}$ computes "the $\infty$-groupoidification" of the homotopy-coherent nerve, more precisely for any locally Kan category $\mA$ and any Kan complex $K$, there is a natural weak homotopy equivalence of mapping spaces $$\Map_{\mathrm{KQ}}\bigl(\mathbb{B}_{\mathrm{KQ}}\mA,K\bigr)\simeq \Map_{\mathrm{Joyal}}\bigl(N^{\mathrm{hc}}\mA,K\bigr) $$ (where $\Map_{\mM}$ denotes the \emph{Dwyer-Kan mapping space} (also called \emph{homotopy function complex}) of the model structure $\mM$).

For a vertical arrow induced by a strict simplicial functor $f:\mA\to\mC$, the realization $ \lVert\mC(c,f-)\rVert_{\mathrm{KQ}} $ is, up to isomorphism in $\Ho\sSet_{\mathrm{KQ}}$, the simplicial set that in degree $n$ is $$\coprod_{a_0,\ldots,a_n} \mC(c,f(a_0))_n \times \mA(a_0,a_1)_n \times\cdots\times \mA(a_{n-1},a_n)_n$$ For a more general vertical arrow $F$ represented by $\widetilde F$, one replaces $\mC(c,f(a_0))_n$ with $\widetilde F(a_0)(c)_n $ in the expression above.

If $\mA$ and $\mC$ are ordinary categories, then $\mathbb B_{\mathrm{KQ}}\mA=N\mA$ and $\lVert\mC(c,f-)\rVert_{\mathrm{KQ}} =N(c\downarrow f)$, so that \Cref{th:enriched-Quillen-A-M} specializes to the classical result of Quillen in its original formulation \cite{Qui73}.

\subsubsection{Joyal's Theorem A} The quasicategorical version of Quillen's theorem \cite[\href{https://kerodon.net/tag/02NY}{Tag 02NY}]{Ker} is also a direct consequence of the simplicial case of \Cref{th:enriched-Quillen-A-M}, using \Cref{thm:comparison-before-strictification}.

\subsubsection{1-categorical finality.} Given ordinary categories $\mA$ and $\mC$, we can instead consider them in the equipment associated to the trivial model structure on the category of sets, to which \ref{rem:locwellp} applies. In formula \ref{eq:realization} in \Cref{ap:A}, we can take $R^n(\cst_{\tu_{\mV}})=\{*\}$ for every $n$, obtaining bijections  $\lVert\mC(c,f-)\rVert_{\Set}  \cong \pi_0 N(c \downarrow f)$ and $\mathbb{B}_{\Set}(\mA) \cong \pi_0(N\mA)$. 

Therefore $f \colon \mA \to \mC$ is final precisely when $c \downarrow f$ is non-empty and connected for every $c$.

\subsubsection{The classifying $1$-type}\label{bifinality}
Let $\mV=\Cat$ endowed with its canonical model structure determined by equivalences of categories and isofibrations. Vertical arrows in our equipment are interpreted as weak $2$-functors, strict arrows are ordinary $2$-functors, formal colimits are bicolimits, and so on. 

Let $\mC$ be a $2$-category. It is convenient to see it as a double category $\mC=(\mC_0,\mC_1)$ with $\mC_1 = \coprod_{(x,y)}\mC(x,y)$ and $\mC_0 = \operatorname{obj}\mC$, so that we can take its \emph{internal nerve}: this is a simplicial object $\Delta^{\mathrm{op}} \to \Cat$ whose category of $n$-simplexes, for a general double category $\mathbb{D}=(\mD_0,\mD_1)$, is defined equal to $\mD_0$ and $\mD_1$ in the 0-th and first degrees, and for $n \geq 2$ by the pullback $N_{int}(\mathbb{D})_n=\mD_1 \times_{\mD_0} \mD_1 \times_{\mD_0} \cdots \times_{\mD_0} \mD_1$ of $n$ copies of $\mD_1$ over $\mD_0$. 
The formulas in \Cref{ap:A}, together with \Cref{def:classifying-functor-M}, say that there is an equivalence of categories $$\mathbb B_{\Cat}\bigl(\mC\bigr) \simeq I^{\bullet}\otimes_{\Delta}N_{int}(\mC)$$ where we have taken as cosimplicial resolution $R^{\bullet}$ (\ref{cosimpres})  the canonical one that in degree $n$ is the groupoid $I^{n}$ with $n+1$ objects each connected by a unique isomorphism.

To proceed further, we recall a slight modification of the double category of squares of an ordinary category $\mD \in \Cat$, we denote it $\operatorname{Sq}_{\widetilde H}(\mD)$. We have $\operatorname{Sq}_{\widetilde H}(\mD)_0=\mD$, while $\operatorname{Sq}_{\widetilde H}(\mD)_1$ is the category of commutative squares in $\mD$ whose top and bottom arrows are invertible, so that, applying the internal nerve to $\operatorname{Sq}_{\widetilde H}(\mD)$, we obtain a simplicial object that is easily recognized as the one having in degree $n$ the category of functors $I^n \to \mD$. 

Now, still seeing $\mC$ as a double category with $\mC_1 = \coprod_{(x,y)}\mC(x,y)$ in the way described at the beginning, we have, for every $\mD \in \Cat$, the natural isomorphisms of hom-sets $$\begin{aligned} \Fun\bigl( I^{\bullet}\otimes_{\Delta}N_{int}(\mC),\mD\bigr)  & \cong \Cat^{\Delta^{\mathrm{op}}}\bigl(N_{int}(\mC),(\mD)^{I^\mathrm{[\bullet]}}\bigr) \\ & \cong \Cat^{\Delta^{\mathrm{op}}}\bigl(N_{int}(\mC),N_{int}\operatorname{Sq}_{\widetilde H}(\mD)\bigr) \\ & \cong \operatorname{DblFun}\bigl(\mC,\operatorname{Sq}_{\widetilde H}(\mD)\bigr) \\ & \cong \Fun \bigl( (\tau_1\mC)^{gp} , \mD\bigr) \end{aligned}$$ where $\tau_1 \colon 2\Cat \to \Cat$ and $(-)^{gp} \colon \Cat \to  \Grpd$ are left adjoints to the respective inclusions,  and the last isomorphism relies on the following fact: in our identification of $\mC$ with a double category, the vertical part, $\operatorname{obj}\mC$, is discrete, so that given a double functor $\mC \to \operatorname{Sq}_{\widetilde H}(\mD)$, a $2$-cell in $\mC$ is sent to a diagram in $\mD$ whose vertical sides are identities and top and bottom side are isomorphisms.

\begin{proposition}
 We have proven that $$\mathbb{B}_{\Cat}(\mC) \simeq (\tau_1\mC)^{\mathrm{gp}}$$ is the \emph{fundamental groupoid} of $\mC$.    
\end{proposition}

\subsubsection{$(2,1)$-categories}\label{(2,1)-finality} The canonical model structure on $\Cat$ restricts to one on $\Grpd$. The associated equipment on groupoid-enriched categories models the formal category theory of weak $(2,1)$-categories. Our formulas do not change, since we can still choose $I^n$ as a cosimplicial resolution, so that $\mathbb B_{\Cat} (\mC) \simeq \mathbb{B}_{\Grpd}(\mC)$. 

Categories enriched in groupoids are naturally seen as locally Kan by a hom-wise application of the ordinary nerve, therefore the considerations of $\Cref{homotopyfinality}$ apply: the classifying space $\mathbb{B}_{\mathrm{KQ}}$ defined there, evaluated on a $\Grpd$-enriched category $\mC$, has the homotopy type of the \emph{Duskin-Street nerve} of $\mC$, and $\mathbb{B}_{\Grpd}$ is the fundamental groupoid of such space. 

\subsubsection{The classifying $2$-type}

Moving up on the categorical ladder, for our equipment to model weak $3$-category theory we have two options: choose as enriching base $\mV = \bigl(2\Cat,\otimes_{Gray}\bigr)$ the category of $2$-categories endowed with Gray tensor product and Lack's model structure \cite{Lac04}, or pick $\bigl(2\Cat_{flex} , \times\bigr)$ the cartesian model category of \emph{flexible $2$-categories} with Campbell's model structure \cite{Cam26}.

As in the previous sections, the only thing one has to do to apply our formulas (\ref{prop:classifying-object-resolution-M}) is to come up with a cosimplicial resolution (\ref{cosimpres}) of the terminal object in $2\Cat_{flex}$. The one we used in $\Cat$ was the cosimplicial groupoid $I^n=\mathrm{[n]}^{\mathrm{gp}}$: this is the value of the right adjoint to the object functor $\mathrm{obj} \colon \Cat \to \Set$ on the set with $n+1$ elements; its natural categorification is therefore given by letting $E^{(-)} \colon \Set \to 2\Cat_{flex}$ denote the right adjoint to the functor $\mathrm{obj}$ \cite{Cam26}. Setting $E^n \coloneqq E^S$ for $S$ a set with $n+1$ elements, we obtain a cosimplicial resolution of $\mathbf{1}$ which works both in $2\Cat_{flex}$ with Campbell's model structure and in $2\Cat_{Gray}$ with Lack's.

\subsubsection{Motivic classifying space}\label{motivicclassifying}

Let $\mV$ be motivic spaces. For a $\mV$-category $\mC$, write $\mC(U)$ for the simplicial category with the same objects and hom-spaces $\mC(c,d)(U)$. Taking $R^n(\cst_{\tu_{\mV}})=\Delta^n$, viewed as a constant simplicial presheaf, \Cref{prop:classifying-object-resolution-M} gives, naturally in $\Ho\mV$, $$\mathbb B_{\mathrm{mot}}\mC \cong L_{\mathrm{mot}}\bigl( U\mapsto\operatorname{diag}N_{\bullet}(\mC(U)_{\bullet}) \bigr).$$ Thus the motivic classifying space is obtained by applying motivic localization to the presheaf of simplicial classifying spaces in \Cref{homotopyfinality}.

\subsubsection{Equivariant classifying space}\label{Equivariantclassifying}

We take $\mV$ to be simplicial $G$-set, for $G$ a discrete group. The model structure here is the \emph{genuine} one, where the weak equivalences and fibrations are detected by mapping out of every orbit. Since every orbit is a connected object, we have, for every $O \in \mO$, an isomorphism of homotopy types $\mathbb{B}_{\mathrm{equiv}}(\mC)^{O} \cong \mathbb{B}_{\mathrm{KQ}}(\mC^O)$ where $\mC^O$ is the simplicial category obtained from $\mC$ whose simplicial hom-set is $\mC^{O}(x,y)=\mC(x,y)^{O}$. It follows that equivariant finality reduces to orbit-wise homotopy finality.

\subsection{Adjunctions}\label{sec:adj}
In this section we study adjunctions of weak $\mV$-functors. 

We remark that, while every vertical arrow $F:\mA\to\mB$ comes, by the equipment property, with a right adjoint $F^{\vee}$ in the horizontal $2$-category (which should be thought of as a dual of $F$) this adjointness does not necessarily take place in $\mV\Cat^h$, since $F^{\vee}$ is not vertical in general.

Let $L:\mA\to\mB$ and $R:\mB\to\mA$ be vertical arrows. An \emph{adjunction} $L\dashv_{\mV\Cat^h}R$ consists of unit and counit $2$-cells $\eta:\id_{\mA}\to RL$ and $\epsilon:LR\to \id_{\mB}$ in the vertical $2$-category such that the triangle identities hold: $(\epsilon L)\circ(L\eta)=\id_L$ and $(R\epsilon)\circ(\eta R)=\id_R.$ The counit corresponds, under closedness (\ref{prop:closed-equipment-M}), to a $2$-cell $$\widehat \epsilon \colon R \to L^{\vee} $$

\begin{lemma}\label{prop:adjunction-exact-square-M}
A $2$-cell $\epsilon:LR\to 1_{\mB}$ is the counit of an adjunction $L\dashv_{\mV\Cat^h}R$ if and only if $\widehat \epsilon  \colon R\to L^\vee$ is invertible.
\end{lemma}

\begin{proof}
If $\epsilon$ is the counit of an adjunction in ${\mV\Cat^h}$, the equipment pseudofunctor ${\mV\Cat^h}\to \Pro_{\mV}$ carries it to an adjunction in $ \Pro_{\mV}$; uniqueness of right adjoints implies that the comparison map $\widehat \epsilon \colon R \to L^{\vee}$ is invertible. Conversely, if $\widehat \epsilon \colon R \to L^{\vee}$ is invertible, we have an adjunction $L \dashv R$ in $\Pro_{\mV}$. By local fully faithfulness we conclude that it takes place in ${\mV\Cat^h}$.
\end{proof}

\begin{proposition}\label{cor:vertical-right-adjoint-M}
A vertical arrow $L:\mA\to\mB$ admits a right adjoint in $\mV\Cat^h$ if and only if its conjoint $L^\vee:\mB\nrightarrow\mA$ is a vertical arrow.
\end{proposition}

\begin{proof}
If $L\dashv R$, then \Cref{prop:adjunction-exact-square-M} gives an isomorphism $R\cong L^\vee$, so $L^\vee$ is representable. Conversely, suppose that $L^\vee$ is representable, and choose a vertical arrow $R:\mB\to\mA$ and an isomorphism $\phi:R\xrightarrow{\cong}L^\vee$. By local full faithfulness, the composite $LR\xrightarrow{L\varphi}LL^\vee\to \id_{\mB}$ is induced by a $2$-cell $\epsilon:LR \to \id_{\mB}$ in the vertical $2$-category. Since $\widehat \epsilon = \phi$, we conclude by \Cref{prop:adjunction-exact-square-M}.
\end{proof}

\subsubsection{Weak adjunctions in $\mV$-category theory}

Let $L:\mA\to\mB$ be a vertical arrow, that is $L \cong \Lan_{\y_{\mA}}\widetilde L$ for a $\mV$-functor $\widetilde L: \mA\to\widetilde{\y(\mB)} \subset\Psh_{\mV}(\mB).$ Let $G:\Psh_{\mV}(\mB)\to\Psh_{\mV}(\mA)$ be the right Quillen adjoint of $L$. Remember that $L^\vee=\overline{G\y_{\mB}}$, and that $G$ need not be a vertical arrow of the equipment.
\begin{proposition}[Representability criterion for weak adjointness]\label{prop:vertical-right-adjoint-representability-M} In the above notation, $L$ admits a right adjoint in $\mV\Cat^h$ if and only if, for every $b\in\mB$ the presheaf $G\y_{\mB}(b)$ is isomorphic in $\Ho\Psh_{\mV}(\mA)$ to a representable.
\end{proposition}
\begin{proof}
This is a rephrasing of \Cref{cor:vertical-right-adjoint-M}
\end{proof}

Let $R:\mB\to\mA$ be a vertical arrow induced (without loss of generality) by a projectively cofibrant functor $\widetilde R:\mB\to\widetilde{\y(\mA)} \subset \operatorname{Psh}_{\mV}(\mA)$. Let $$\epsilon:LR\to \id_{\mB}$$ be a $2$-cell with $G$ denoting the right Quillen adjoint of $L$ as above. After restriction along $\y_{\mB}$, and under the adjointness $L\dashv G$, the morphism $\epsilon$ determines a morphism $[\epsilon^\sharp]: \widetilde R\to G\y_{\mB}$ in $\Ho\bigl( \mV\Fun (\mB,\Psh_{\mV}(\mA))_{\mathrm{proj}} \bigr)$. Since we assumed $\widetilde R$ cofibrant, may choose an enriched natural transformation $$\epsilon^\sharp:\widetilde R\to G\y_{\mB}$$ representing it.
\begin{lemma}\label{prop:adjunction-model-M}
The $2$-cell $\epsilon$ is the counit of an adjunction $L\dashv_{\mV\Cat^h}R$ if and only if, for every $a\in\mA$ and $b\in\mB$, the morphism $$\epsilon^\sharp_{a,b}: \widetilde R(b)(a) \to \Psh_{\mV}(\mB) \bigl(\widetilde L(a),\y_{\mB}(b)\bigr)$$ is a weak equivalence in $\mV$.
\end{lemma}

\begin{proof}
By \Cref{prop:adjunction-exact-square-M} it suffices to determine when $\widehat \epsilon $ is invertible. Let $\lambda_L:L^\vee\y_{\mB}\xrightarrow{\cong}G\y_{\mB}$ be the canonical isomorphism that exists by definition of $L^{\vee}$. Similarly $\lambda_R \colon R\y_{\mB}\cong\widetilde R$ by definition of $\widetilde R$. We have a commutative triangle $$\begin{tikzcd}[column sep=4.5em,row sep=3em] R\y_{\mB} \ar[r,"{\widehat{\epsilon} \y_{\mB}}"] \ar[dr,"{[\epsilon^\sharp]}"'] & L^\vee\y_{\mB} \ar[d,"\lambda_L","\cong"'] \\ & G\y_{\mB}. \end{tikzcd}$$ Hence $\widehat \epsilon$ is invertible if and only if (\ref{lem:homotopical-Yoneda-comparison-M}) $\widehat \epsilon\y_{\mB}$ is invertible if and only if $[\epsilon^\sharp]$ is invertible.
\end{proof}

\subsubsection{Weak adjointness of strict $\mV$-functors}\label{sec:adjstrict}

\begin{proposition}\label{cor:right-adjoint-strict-M}
Let $f:\mA\to\mB$ be a strict $\mV$-functor. The vertical arrow $f_!$ admits a right adjoint if and only if $\mB(f-,b)$ is isomorphic in $\Ho\Psh_{\mV}(\mA)$ to a representable for every $b\in\mB$.
\end{proposition}

\begin{proof}
Follows by \ref{prop:vertical-right-adjoint-representability-M} noting that the right Quillen adjoint of $f_!$ is $\res_f$, and $\res_f\y_{\mB}(b)=\mB(f-,b)$.
\end{proof}

\begin{proposition}\label{cor:adjunction-strict-M}
Let $f:\mA\to\mB$ and $g:\mB\to\mA$ be strict $\mV$-functors, and let $\epsilon:fg\to \id_{\mB}$ be an enriched natural transformation. The induced $2$-cell $f_!g_!\to \id_{\mB}$ is the counit of an adjunction $f_!\dashv_{\mV\Cat^h}g_!$ if and only if $$\mA(a,gb) \xrightarrow{\,f\,} \mB(fa,fgb) \xrightarrow{\,\mB(fa,\epsilon_b)\,} \mB(fa,b)$$ is a weak equivalence in $\mV$ for every $a\in\mA$ and $b\in\mB$.
\end{proposition}

\begin{proof}
Follows by \Cref{prop:adjunction-model-M}.
\end{proof}

\subsection{A Whitehead theorem for $\mV$-categories}\label{sec:whitehead}

\begin{definition} [\cite{Woo82}]\label{def:fully-faithful-arrow-M} A vertical arrow $F:\mA\to\mB$ is \emph{fully faithful} if the unit $\eta_F \colon \id_{\mA} \to F^{\vee}F$ of its companion-conjoint adjunction is invertible.

\end{definition}

Let $F:\mA\to\mB$ be induced by a $\mV$-functor $\widetilde{F}:\mA\to\widetilde{\y(\mB)} \subset \Psh_{\mV}(\mB)$, and let $G$ be the right Quillen adjoint of $F=\Lan_{\y_{\mA}}\widetilde{F}$.

\begin{proposition}\label{prop:fully-faithful-arrow-M}
The vertical arrow $F$ is fully faithful if and only if, for every $a,a'\in\mA$, the morphism $$\mA(a,a') \to \Psh_{\mV}(\mB)\bigl(\widetilde{F}(a),\widetilde{F}(a')\bigr)$$ induced by $\widetilde{F}$ is a weak equivalence in $\mV$.
\end{proposition}

\begin{proof}
As always, using \ref{lem:homotopical-Yoneda-comparison-M}, in order to check the invertibility of $\eta_F:\id_{\mA}\to F^\vee F$ we restrict on $\y_{\mA}$. This restriction  is constructed in the proof of \Cref{th:equipment-M} from the morphism $u_{\widetilde{F}}:\y_{\mA}\to G\widetilde{F}$ obtained by restriction on $\y_{\mA}$ of the unit of the Quillen adjunction $F \dashv G$.
\end{proof}

\begin{corollary}\label{cor:fully-faithful-strict-M}
Let $u:\mA\to\mB$ be a strict $\mV$-functor. The associated vertical arrow $u_!$ is fully faithful if and only if $$\mA(a,a') \to \mB(ua,ua')$$ is a weak equivalence in $\mV$ for every $a,a'\in\mA$.
\end{corollary}

\begin{proof}
Follows by taking $F=u_{!}$ and $\widetilde{F}=\y_{\mB}u$ in \Cref{prop:fully-faithful-arrow-M}.
\end{proof}

\begin{definition}\label{def:essentially-surjective-arrow-M}
We define a vertical arrow $F:\mA\to\mB$ to be \emph{essentially surjective} when there is an $\widetilde{F}:\mA\to\widetilde{\y(\mB)} \subset \Psh_{\mV}(\mB) $ inducing it such that for every $b\in\mB$, there exist $a\in\mA$ and an isomorphism $\widetilde{F}(a)\cong\y_{\mB}(b)$ in $\Ho\Psh_{\mV}(\mB)$.
\end{definition}

\begin{proposition}\label{prop:vertical-equivalence-M}
For a vertical arrow $F:\mA\to\mB$ the following are equivalent: \begin{itemize}
    
\item[(i)] $F$ is a fully faithful and essentially surjective arrow (\ref{def:fully-faithful-arrow-M},\ref{def:essentially-surjective-arrow-M});
\item[(ii)] $F$ is an adjoint equivalence in the $2$-category $\mV\Cat^h$.
\end{itemize} 
\end{proposition}

\begin{proof}
Suppose first that $F$ is an adjoint equivalence, and let $H:\mB\to\mA$ be an inverse, induced by $\widetilde{H}:\mB\to\widetilde{\y(\mA)}$. We have $H\cong F^{\vee}$ so that $F$ is fully faithful. In order to see that it is essentially surjective, we restrict the isomorphism $FH\cong \id_{\mB}$ along $\y_{\mB}$ obtaining $\y_{\mB}(b) \cong FH\y_{\mB}(b) \cong F\widetilde{H}(b)$, and we conclude by choosing for every $b \in \mB$ an element $a\in\mA$ and an isomorphism $\widetilde{H}(b)\cong\y_{\mA}(a)$. In the other direction, suppose that $F$ is fully faithful and essentially surjective, then the unit of $F\dashv F^\vee$ is invertible by definition. Regarding the counit $\epsilon:FF^\vee\to \id_{\mB}$, the triangle identity implies that it is invertible after restriction to every object of the form $F\y_{\mA}(a)$, and hence to every $\widetilde{F}(a)$. By essential surjectivity, it is therefore invertible after restriction to every representable $\y_{\mB}(b)$, which means (\ref{lem:homotopical-Yoneda-comparison-M}) that it is invertible. Since $\mV\Cat^h$ is locally full in $\Pro_{\mV}$ the unit and counit belong to the vertical $2$-category as long as $F^{\vee}$ is a vertical arrow. We check this. For every $a\in\mA$, \Cref{prop:fully-faithful-arrow-M} gives an isomorphism $\y_{\mA}(a)\cong G\widetilde{F}(a)$ in $\Ho\Psh_{\mV}(\mA)$. Given $b\in\mB$, we choose $a\in\mA$ and an isomorphism $\y_{\mB}(b)\cong \widetilde{F}(a)$, so that $G\y_{\mB}(b)\cong\y_{\mA}(a)$, from which we obtain that $F^\vee=\overline{G\y_{\mB}}$ is a vertical arrow. 
\end{proof}

\begin{definition}[\cite{DK80}]
  A strict $\mV$-functor $u:\mA\to\mB$ is a \emph{Dwyer-Kan equivalence} if $$\mA(a,a')\to\mB(ua,ua')$$ is a weak equivalence in $\mV$ for every $a,a'\in\mA$, and the induced functor $$u_{\pi_0}:\mA_{\pi_0}\to\mB_{\pi_0}$$ is essentially surjective, where $\mA_{\pi_0}(a,a') \coloneqq (\Ho\mV)\bigl(\tu_{\mV},\mA(a,a')\bigr)$ (and $\operatorname{obj} \mA_{\pi_0} \coloneqq \operatorname{obj} \mA$).
\end{definition}

\begin{corollary}[Whitehead theorem for $\mV$-categories]\label{cor:DK-equivalence-vertical-M}
A strict $\mV$-functor $u:\mA\to\mB$ is a Dwyer--Kan equivalence if and only if the associated vertical arrow $u_!:\mA\to\mB$ is an adjoint equivalence in $\mV\Cat^h$.
\end{corollary}

\begin{proof}
Follows from \Cref{cor:fully-faithful-strict-M} and \Cref{prop:vertical-equivalence-M}.
\end{proof}

\section{Proof of the equivalence with the formal category theory of quasicategories}\label{sec:RV}
\subsection{The Riehl-Verity equipment}

An $\infty$-cosmos is a category strictly enriched in quasi-categories, equipped with a distinguished class of morphisms called \emph{isofibrations} and admitting certain limits \cite[1.2.1]{RV22}. The example of interest for us is the full simplicial subcategory $\Qc\subset \sSet$ spanned by the quasi-categories, with its cartesian closed structure.
\subsubsection{The homotopy $2$-category of an $\infty$-cosmos}

For any $\infty$-cosmos $\mK$, Riehl-Verity define its homotopy $2$-category $\hh\mK$ by changing the enrichment along the homotopy category functor $\h\colon\Qc\to\Cat$; see \cite[1.4.1]{RV22}. This $\hh\mK$ has the same objects as $\mK$, its $1$-cells are the vertices of the functor quasi-categories, and its $2$-cells are homotopy classes of $1$-simplexes in those functor quasi-categories.

\subsubsection{Modules in an $\infty$-cosmos}

\begin{definition}\label{def:twosidedisofibration}
A \emph{two-sided isofibration from $A$ to $B$} in an $\infty$-cosmos $\mK$ is a span $$A\xleftarrow{q}E\xrightarrow{p}B$$ for which the induced morphism $(q,p)\colon E\to A\times B$ is an isofibration.
\end{definition}

A \emph{two-sided fibration} is a two-sided isofibration which is cocartesian in the $A$-variable and cartesian in the $B$-variable, satisfying the compatibility specified in \cite[7.1.4]{RV22}.

\begin{definition}\label{def:module}\cite[7.4.1 and 7.4.2]{RV22} A \emph{module from $A$ to $B$} is a two-sided isofibration $$E\to A\times B$$ which is cocartesian on the left, cartesian on the right, and discrete as an object of $\mK/(A\times B)$. Equivalently, it is a two-sided fibration $A\xleftarrow{q}E\xrightarrow{p}B$ which is discrete as an object of $A\backslash\operatorname{Fib}(\mK)/B$.
\end{definition}

By \cite[12.2.3]{RV22}, an isofibration of quasi-categories is discrete in the slice if and only if all of its fibers are discrete objects of $\Qc$, which means that they are Kan complexes.

Riehl-Verity construct a virtual double category $\Mod(\mK)$ whose objects are those of $\mK$, whose vertical arrows are the functors in $\mK$, and whose horizontal arrows are modules \cite[8.1.14]{RV22}.

\begin{theorem}\cite[8.2.6] {RV22}\label{thm:virtual-equipment-modules} For every $\infty$-cosmos $\mK$, the virtual double category $\Mod(\mK)$ of modules is a virtual equipment.
\end{theorem}

Specializing to $\mK=\Qc$, we now identify modules with maps of simplicial sets $B^{\mathrm{\op}}\times A \to \mS$ and prove that horizontal compositions exist in the sense of \cite[ 8.3.1]{RV22}, so that we obtain an equipment in the sense of definition \ref{def:woodequip}.

Let $\Cat_{\infty}$ denote the quasi-category obtained from the simplicial category $\Qc$ by applying the core functor to each hom quasi-category and then the homotopy coherent nerve to the resulting simplicial category. Let $\mS \subset\Cat_{\infty}$ be the full sub-quasi-category spanned by the Kan complexes. For quasi-categories $A$ and $B$, let $$\operatorname{Mod}(A,B)$$ denote the quasi-category presented by Stevenson's bivariant model structure on $\sSet_{/A\times B}$ \cite{Ste18}. Its fibrant objects are precisely his bifibrations, which are the modules of Riehl-Verity.

Stevenson proves that a map between bifibrations is a bivariant equivalence if and only if it is a weak homotopy equivalence on every fiber \cite[4.27]{Ste18}.

\begin{proposition}[]\label{prop:stevenson-modules}
For any quasi-categories $A$ and $B$, there is an equivalence, natural in both variables, of quasi-categories $$\operatorname{Mod}(A,B) \simeq \Fun(B^{\op}\times A,\mS).$$

\end{proposition}

\begin{proof}
For quasi-categories $A$ and $B$, Stevenson constructs a zigzag of Quillen equivalences $$\bigl(\sSet_{/B^{\op}\times A}\bigr)_{\mathrm{cov}} \simeq_{} \operatorname{Corr}(A,B) \simeq_{} \bigl(\sSet_{/A\times B}\bigr)_{\mathrm{biv}};$$ see \cite[3.23 and 4.40]{Ste18}. Passing to the quasi-categories presented by these model categories shows that $\operatorname{Mod}(A,B)$ is the quasi-category presented by the covariant model structure on $\sSet_{/B^{\op}\times A}$. The straightening theorem for left fibrations (the dual of \cite[2.2.1.2]{HTT}) recognizes the latter as $\Fun(B^{\op}\times A,\mS)$. For naturality, see \Cref{rem:natural-modules-profunctors}.
\end{proof}
\begin{remark}\label{rem:natural-modules-profunctors}
The equivalences of \Cref{prop:stevenson-modules} may be chosen naturally in both variables, see \cite[6.18, 6.19 and A.2]{HHLN23}. 
\end{remark}

\subsubsection{Horizontal composition of modules}

The pullback of two composable modules need not itself be a module, see \cite[7.4.ii]{RV22}. It is, however, a two-sided fibration by \cite[7.2.6]{RV22}. The existence of horizontal composites will follow by reflecting two-sided fibrations into modules.

\begin{lemma}\label{lem:stevenson-module-reflection}
Let $T\to A\times C$ be a two-sided fibration in an $\infty$-cosmos. Suppose that for any such $T$ there is a module $\overline T\to A\times C$ and a map $$\eta_T\colon T\to\overline T$$ over $A\times C$, such that precomposition with $\eta_T$ induces a homotopy equivalence of Kan complexes $$\Map_{A\times C}(\overline T,H) \xrightarrow{\;\simeq\;} \Map_{A\times C}(T,H)$$ for every module $H\to A\times C$. Then the virtual equipment of modules admits all composites.
\end{lemma}
\begin{proof}
This is \cite[8.3.ii(ii)]{RV22}.
\end{proof}

\begin{corollary}
\label{thm:actual-equipment-of-modules}
The virtual equipment $\Mod(\Qc)$ admits all horizontal composites, and is therefore an equipment according to definition \ref{def:woodequip}.
\end{corollary}

\begin{proof}
\Cref{prop:stevenson-modules} and the fibrant replacement in Stevenson's bivariant model structure (\cite[Theorem 4.19]{Ste18}), together with its definition of the weak equivalences (\cite[Definition 4.12]{Ste18}), allow us to apply \Cref{lem:stevenson-module-reflection}.
\end{proof}

\subsection{An intermediate equipment of quasicategorical profunctors}
\begin{definition}\cite[\href{https://kerodon.net/tag/03MB}{03MB}]{Ker} A \emph{profunctor} $P\colon A\nrightarrow B$ from a quasi-category $A$ to $B$ is a map of simplicial sets $P\colon B^{\op}\times A\to\mS $.
\end{definition}

\begin{definition}
 Let $\mP(A)=\Fun(A^{\op},\mS)$. A left adjoint $L:\mP(A)\to\mP(B)$ will be called \emph{representable} if there exist a functor $f:A\to B$ and an equivalence $L\y_A\simeq\y_B f.$
\end{definition}

By the quasicategorical universal property of $\mP(A)$ \cite[5.1.5.6]{HTT}, restriction along $\y_A$ induces an equivalence $$\Fun^L(\mP(A),\mP(B)) \xrightarrow{\;\simeq\;} \Fun(A,\mP(B)) \cong \Fun(A \times B^{\op},\mS).$$
\begin{notation}\label{notation:overlinequasicat}
  We denote a chosen quasi-inverse to the above equivalence by $$\overline{(-)} \colon \Fun(A \times B^{\op},\mS) \to \Fun^L(\mP(A),\mP(B)).$$
\end{notation}
Under this equivalence, a \emph{representable} $L$, represented by $f \colon A \to B $, corresponds to the quasicategorical left Kan extension $\overline{\y_Bf}:\mP(A)\to\mP(B)$.

The functor $\overline{\y_Bf}$ is left adjoint to the restriction functor $f^*:\mP(B)\to\mP(A)$ \cite[4.3.3.7]{HTT}. Under currying, $\y_Bf$ is the profunctor $$(b,a) \mapsto \Map_B(b,f(a))$$ while $f^*\y_B$ corresponds to the profunctor from $B$ to $A$ given by $$(a,b) \mapsto \Map_B(f(a),b).$$

\begin{definition}
 Composition of left adjoints defines a $2$-category $q Pro$ whose objects are all small quasi-categories, and whose hom-categories are $$qPro(A,B) = \operatorname{h}\bigl(\Fun^L\bigl(\mP(A),\mP(B)\bigr)\bigr).$$ The identity $1$-cell of $A$ is $\id_{\mP(A)}$.
   
\end{definition}

 Let $Rep$ be the locally full sub-$2$-category of $qPro$ with the same objects and with the hom-categories spanned by the representable left adjoints.

\begin{proposition}
\label{prop:presheaf-equipment}
The inclusion $$Rep \to qPro$$ is an equipment in the sense of definition \ref{def:woodequip}.
\end{proposition}

\begin{proof}
 Suppose that $L:\mP(A)\to\mP(B)$ is representable, that is $L \simeq \overline{\y_Bf}$ for some $f:A\to B$. To check the equipment property we need to construct the required right adjoint $f^{\vee}$ to $\overline{\y_Bf}$ in $qPro$. It therefore suffices to show that the right adjoint $f^*$ of $\overline{\y_Bf}$ is an arrow of $qPro$. This holds since restriction preserves small colimits and therefore belongs to $qPro(B,A)$ by the quasicategorical adjoint functor theorem \cite[5.5.2.9]{HTT}.
\end{proof}

\begin{definition}
Let $\bE_{qP}$ be the double category associated to the equipment \ref{prop:presheaf-equipment}.  Its objects are small quasi-categories, its vertical arrows are representable left adjoints, and its horizontal arrows are arbitrary left adjoints between presheaf quasi-categories.  A square $$\begin{tikzcd} A \arrow[r,"P"] \arrow[d,"V"'] & B \arrow[d,"K"] \\ A' \arrow[r,"M"'] & B' \end{tikzcd}$$ is a morphism $K\circ P\to M\circ V$ in $\operatorname{h}\Fun^L(\mP(A),\mP(B')).$
\end{definition}

\begin{lemma}
\label{lem:kernel-square-correspondence}
Let $f:A\to A'$, $g:B\to B'$, $P:B^{\op}\times A\to\mS$, and $Q:(B')^{\op}\times A'\to\mS$.

There is a natural equivalence of mapping spaces: $$\Map_{\Fun^L(\mP(A),\mP(B'))} \bigl(\overline{\y_{B'} g} \circ \overline{ P},\overline{Q} \circ \overline{\y_{A'} f}\bigr) \simeq \Map_{\Fun(B^{\op}\times A,\mS)} \bigl(P,Q\circ(g^{\op}\times f)\bigr).$$ 
\end{lemma}
\begin{proof}
Follows from the adjunction $\overline{\y_{B'}g} \dashv g^*$, the universal property of presheaves \cite[5.1.5.6]{HTT}, and the currying isomorphism, using $\overline{\y_{A'}f}\y_A\simeq\y_{A'}f$.
\end{proof}
\subsection{Comparison of the latter with the equipment of simplicial categories} Let $\bE_{\Delta}$ denote the double category associated to the equipment constructed in \ref{th:equipment-M} for $\mV = \sSet_{\mathrm{KQ}}$. We abbreviate the notation by writing $\Pro_{\sSet}= \Pro_{\Delta}$ and $\sSet\Cat^{h} =\Cat^{h}_{\Delta}$.

\begin{proposition}
\label{prop:simplicial-presheaf-comparison}
For small locally Kan simplicial categories $\mA$ and $\mB$, there is an equivalence of categories $$\Phi_{\mA,\mB}: \Pro_{\Delta}(\mA,\mB) \xrightarrow{\;\simeq\;}        \operatorname{h}\Fun^L \bigl(\mP(N^{\mathrm{hc}}\mA),\mP(N^{\mathrm{hc}}\mB)\bigr).$$
\end{proposition}

\begin{proof}
Lemma \ref{lem:homotopical-Yoneda-comparison-M}, specialized to simplicial sets with the Kan-Quillen model structure, gives $$\Pro_{\Delta}(\mA,\mB) \simeq \Ho\bigl( \sSet\Fun (\mA,\Psh_{\sSet}(\mB))_{\mathrm{proj}} \bigr).$$ Then, using \cite[4.2.4.4]{HTT}, we recognize the right-hand side as the category $\operatorname{h}\Fun(N^{\mathrm{hc}}\mA,\mP(N^{\mathrm{hc}}\mB)),$ and \cite[5.1.5.6]{HTT} ensures the latter is equivalent to $\operatorname{h}\Fun^L \bigl(\mP(N^{\mathrm{hc}}\mA),\mP(N^{\mathrm{hc}}\mB)\bigr).$ 
\end{proof}
At the level of connected components, $\Phi_{\mA,\mB}(F)$ is obtained by left Kan extending the underlying quasicategorical functor associated to $F\y_A$ along the quasicategorical Yoneda embedding; therefore it sends representable Quillen profunctors to representable quasicategorical profunctors. Finally, the factorization through Yoneda property defining representability is reflected by construction, once we know that $\Phi$ is an equivalence and therefore bijective on isomorphism classes of objects.
\begin{proposition}\label{prop:simplicial-equipment-comparison}
The equivalences of \Cref{prop:simplicial-presheaf-comparison} assemble into a morphism of equipments $$\Phi:\bE_{\Delta}\to\bE_{qP},$$ where $\Phi(\mA)=N^{\mathrm{hc}}\mA$.
\end{proposition}

\begin{proof}
In order to exhibit a morphism of equipments it suffices to check pseudofunctoriality on the horizontal fragment. The latter is ensured by the quasicategorical universal property of the presheaf construction. In fact it suffices to note that $\Phi(GF)$ and $\Phi(G)\Phi(F)$ are the left Kan extensions of the same diagram given by the quasicategorical functor underlying $GF\y_{\mA}$. Therefore we obtain coherent isomorphisms giving pseudofunctoriality.
\end{proof}
\begin{proposition}
\label{prop:Phi-Delta-Verdugo-conditions}
The comparison $\Phi_{}$ satisfies $(w1)$, $(w2)$, $(w3')$, and $(w4)$ of \ref{def:equipequivalence}.
\end{proposition}

\begin{proof}
For $(w1)$, it is a Theorem of Joyal that for any quasi-category $A$ there is a locally Kan simplicial category $\mA$ and a Joyal weak equivalence $A\simeq N^{\mathrm{hc}}\mA$ \cite[2.2.5.1]{HTT}, which gives a vertical equivalence in $\bE_{qP}$.
The conditions $(w2)$ and $(w3')$, as well as $(w4)$ all hold by construction of $\Phi$ and \Cref{prop:simplicial-presheaf-comparison}.
\end{proof}

\subsection{Comparison with the Riehl--Verity equipment}
Let $\bE_{RV}$ be the equipment of modules (\ref{thm:actual-equipment-of-modules}) associated to Riehl-Verity's $\infty$-cosmos of quasicategories. We construct a morphism $$\Psi \colon \bE_{RV} \to \bE_{qP}$$ which is the identity on objects and on arrows is defined in the following way. For a module $E:A\nrightarrow B$ we take $P_E:B^{\op}\times A\to\mS$ the corresponding profunctor (\ref{prop:stevenson-modules}) and we set $$\Psi(E):=\overline{P_E}: \mP(A)\to\mP(B),$$ following notation \ref{notation:overlinequasicat}. In particular, on a vertical arrow it is $\Psi(f)=\overline{\y f}$; on squares, we use \Cref{rem:natural-modules-profunctors,lem:kernel-square-correspondence}.

\begin{lemma}\label{lem:module-comparison-coherence}
The assignment above defines a morphism of equipments.
\end{lemma}
\begin{proof}
As in the other case, it is sufficient to check pseudofunctoriality on the horizontal fragment. The composition of modules in $\mathbb{E}_{\mathrm{RV}}$ is given by taking their composite as two-sided fibrations and then reflecting the result onto modules. In quasicategories this means taking the coend of the associated profunctors, by \cite[5.2.1(3)]{AF20}. The quasicategorical Fubini's theorem for coends \cite[\href{https://kerodon.net/tag/06A2}{06A2}]{Ker} and the universal property of the presheaf construction then ensure the pseudofunctoriality of our definition.
\end{proof}
\begin{proposition}
\label{prop:Phi-RV-Verdugo-conditions}
The comparison $\Psi$ satisfies $(w1)$, $(w2)$, $(w3')$, and $(w4)$ of \ref{def:equipequivalence}.
\end{proposition}

\begin{proof}
Let $L:\mP(A)\to\mP(B)$ be a left adjoint; then by \Cref{prop:stevenson-modules} there is a module with associated profunctor given by (the currying of) $L\y_A$; the equivalence $\overline{L\y_A} \simeq L$ thus gives both $(w3')$ and $(w2)$, whereas $(w4)$ follows by \Cref{lem:kernel-square-correspondence}.
\end{proof}

\subsection{Comparison with topological categories}
Let $\mathsf{Top}$ denote compactly generated weakly Hausdorff spaces with the Quillen-Serre model structure, and let $\bE_{\mathsf{Top}}$ be the equipment constructed in \Cref{th:equipment-M} for this base.

\begin{proposition}\label{prop:simplicial-topological-equipment}
Geometric realization induces an equipment biequivalence $$|-|\colon \bE_{\Delta}\to\bE_{\mathsf{Top}}.$$
\end{proposition}

\begin{proof}
We let $|\mC|$ be the topological category obtained by applying the geometric realization functor locally to the simplicial mapping sets of $\mC$. We recall the classical fact that geometric realization is left adjoint to the singular functor, and as such is part of a Quillen equivalence. As always, we only need to define the morphism of equipments pseudo-functorially on the horizontal fragment. 

Let $H \colon \mA \nrightarrow \mB$ a horizontal arrow in $\mathbb{E}_{\Delta}$. Taking $\mC = \mA^{\mathrm{\op}} \times \mB$, we use that realization extends to a left Quillen equivalence $|-|\colon \Psh_{\sSet}(\mC) \to  \Psh_{\mathsf{Top}}(|\mC|)$, so that we can define $|H|=\Lan_{\y_{|\mA|}}|H\y_{\mA}|$, thus obtaining a pseudofunctor. We are left to show that it is an equivalence of equipments. The conditions $(w2),(w3')$ are met owing to the Quillen equivalence $\Psh_{\sSet}(\mC) \to  \Psh_{\mathsf{Top}}(|\mC|)$, and local full faithfulness gives $(w4)$. Finally, the map $|\operatorname{Sing} \mC |\to \mC$ is locally a weak homotopy equivalence and an identity on object, therefore we have also $(w1)$.

\end{proof}

\subsection{Model invariance}
\begin{theorem}
\label{thm:comparison-before-strictification}
There are equivalences of equipments (definition \ref{def:equipequivalence}) $$\mathbb{E}_{\mathsf{Top}}\xleftarrow{\;|-|\;} \bE_{\Delta} \xrightarrow{\;\Phi\;} \bE_{qP} \xleftarrow{\;\Psi\;} \bE_{\mathrm{RV}}$$

\end{theorem}

\begin{proof}
Follows from \Cref{prop:Phi-Delta-Verdugo-conditions}, \Cref{prop:Phi-RV-Verdugo-conditions} and \Cref{prop:simplicial-topological-equipment}.
\end{proof}

\begin{corollary}[Model invariance I]\label{cor:modelinvariance}
   Formulas in Riehl-Verity's language for formal category theory are invariant under the equivalences of \Cref{thm:comparison-before-strictification}.
\end{corollary}
\begin{proof}
    Follows by \Cref{prop:RVstructureequiv} and \cite[11.2.28]{RV22}.
\end{proof}
Let $\mathbb{E}^{\mathrm{str}}_{\mathrm{RV}} \to \mathbb{E}_{\mathrm{RV}}$ be the horizontal strictification of $\mathbb{E}_{\mathrm{RV}}$. We obtain equivalences where all equipments are strict $$\mathbb{E}_{\mathsf{Top}}\xleftarrow{\;|-|\;} \bE_{\Delta} \xrightarrow{\;\Phi\;} \bE_{qP} \xleftarrow{\;\Psi'\;} \mathbb{E}^{\mathrm{str}}_{\mathrm{RV}}$$

\begin{corollary}[Model invariance II]
    Formulas in Verdugo's language for formal category theory are invariant under the equivalences $|-|$, $\Phi$ and $\Psi'$.
\end{corollary}
\begin{proof}
Follows from \Cref{Verdugostructureequiv} and \cite[6.32]{Ver25}.
\end{proof}

\section{Addendum: the enhancement to an $\infty$-equipment}\label{sec:inftyequip}

We shall see that the equipment constructed in \Cref{th:equipment-M} is a quotient of an $\bigl(\infty,2\bigr)$-categorical equipment. 

For our purpose it is convenient to adopt as a model of $(\infty,2)$-categories the one given by complete twofold Segal spaces \cite{Bar05}. We recall that a twofold Segal space is a bisimplicial object in Kan complexes satisfying the Segal condition in each variable, and such that $X_{0,\bullet}$ is essentially discrete. It is called \emph{locally complete} when $X_{p,\bullet}$ is complete for every $p$, and \emph{complete} when in addition $X_{\bullet,0}$ is complete.

The construction \ref{def:Pro-M-hom} upgrades to a category $\Pro^{rel}_{\mV}$ enriched in relative categories, simply by the compatibility of the composition of left Quillen functors with cofibrant-wise equivalences. We let $\Cat_{\infty}$ denote the very large quasicategory of possibly large quasicategories, and stipulate that $\mS$ in this section denotes large spaces. Localization preserves coproducts and finite products as a functor $L_{DK} \colon N(\Cat_{rel}) \to \Cat_{\infty}$ from the very large $1$-category of relative categories. Given that $\Cat_{rel}$ is cartesian, we can consider the internal nerve of $\Pro^{rel}_{\mV}$, producing a simplicial object $N_{\Cat_{rel}}(\Pro^{rel}_{\mV}) \colon \Delta^{\op} \to \Cat_{rel}$. By composition we therefore obtain a functor $$\N_{\bullet,\bullet}(\Pro^{rel}_{\mV}) \colon N(\Delta^{\op}) \to N(\Cat_{rel}) \xrightarrow{L_{DK}} \Cat_{\infty} \hookrightarrow \Fun(\Delta^{\op},\mS)$$
That is, the large Kan complex $$\N_{p,q}(\Pro^{rel}_{\mV}) = \Fun\bigl(\mathrm{[q]},L_{DK}\bigl(N_{\Cat_{rel}}(\Pro^{rel}_{\mV})_p\bigr)\bigr)^{\simeq}$$ is given by the core (the maximal Kan subcomplex) of the quasicategory of functors from $\mathrm{[q]}$ to the localization of the relative category of strings of length $p$ of composable Quillen profunctors.

In the $q$-direction the Segal maps are equivalences by construction; in the $p$-direction, the Segal condition follows from the limit-preservation of $\Fun(\mathrm{[q]},-)^{\simeq}$. We obtain a locally complete two-fold Segal space. We now restrict, in each degree $p$, to the full replete subcategory spanned by strings of representable Quillen profunctors, and then apply the two-fold completion functor to the resulting inclusion of locally complete two-fold Segal spaces. This way, we get a map of complete two-fold Segal spaces that we denote $$\iota \colon \mV\Cat_{\infty} \to \Pro^{\infty}_{\mV}$$
This morphism, by construction, is locally fully faithful and every object in the target is, up to equivalence, in its image. Since adjunctions in complete two-fold Segal spaces are detected at the level of the homotopy $2$-category (\cite[11.2 and 11.5]{Hau18}, following \cite[4.3.11]{RV16}) our $\Cref{th:equipment-M}$ ensures that the morphism $\iota$ satisfies the $\infty$-categorical version of Wood's definition of equipments. The corresponding version in terms of double $\infty$-categories has been introduced by Ruit \cite{Rui23}. It is readily seen that $\iota$ gives rise to a double $\infty$-category, so that it is an equipment in Ruit's sense \cite[5.18 and 5.19]{Rui23}.

\appendix

\section{Enriched projectively cofibrant replacements}
\label{ap:A}
\label{subsec:projective-cofibrant-replacements-M}

Let $\mV$ a monoidal model category with cofibrant unit, let $\mC$ a locally cofibrant $\mV$-category, and assume the projective model structure exists on $\mV\Fun(\mC,\mV)$. Throughout this section, for an object $X$ in a model category, a cofibrant replacement $QX$ of $X$ means cofibrant object weakly equivalent to $X$.

We record an enrichedly functorial projectively cofibrant replacement for enriched diagrams: it is the one obtained from the comonad resolution associated with the adjunction $$T: \mV\Fun(\operatorname{\mathrm{obj}}\mC,\mV) \rightleftarrows \mV\Fun(\mC,\mV):U$$ where $(TZ)(c) = \coprod_{c_0\in\operatorname{obj}\mC} \mC(c_0,c)\otimes Z(c_0)$, and $\operatorname{\mathrm{obj}}\mC$ is the discrete $\mV$-category with the same objects as $\mC$ (defined by ${\mathrm{obj}}\mC(c,c') =\tu_{\mV}$ when $c=c'$ and the initial object otherwise).

\begin{notation}\label{not:B-bullet}
    For a $\mV$-functor $X:\mC\to\mV$, we denote $$B_\bullet^{\mC}X\to X$$ its comonad resolution, which is an augmented simplicial object $\Delta_+^{\op} \to \mV\Fun\bigl(\mC,\mV\bigr)$ given by $$B_n^{\mC}X=(TU)^{n+1}X$$ and $B_{-1}^{\mC}X=X$. Explicitly, for $n \geq 0$
\begin{equation}\label{eq:cotriple-resolution-M} B_n^{\mC}X(c) = \coprod_{c_0,\ldots,c_n} \mC(c_n,c)\otimes \mC(c_{n-1},c_n)\otimes\cdots\otimes \mC(c_0,c_1)\otimes X(c_0). \end{equation}
\end{notation}
Recall that a cosimplicial object $C \colon \Delta \to \mV$ in a model category is \emph{Reedy cofibrant} when the latching maps $L_n \colon\partial^n C \to C_n $ are cofibrations in $\mV$, where $\partial^n C = \colim_{ [m]<[n] \in \Delta_{\mathrm{inj}}}C^m$. \begin{definition}\label{cosimpres} A \emph{cosimplicial resolution} of an object $X \in \mV$ is a Reedy cofibrant $R(X) \colon \Delta \to \mV$ which is weakly equivalent to the constant cosimplicial object at $X$.
 \end{definition}
\begin{notation}\label{not:realization}
    Let $R^{\bullet}(\cst\tu_{\mV})\xrightarrow{\sim}\cst \tu_{\mV}$ be a Reedy cofibrant replacement of the constant cosimplicial object at the monoidal unit of $\mV$ . Let $\mN$ be a $\mV$-model category. Let $\Delta_{\mathrm{inj}}\subseteq\Delta$ be the subcategory of injective maps. For a simplicial object $Z_\bullet$ in $\mN$, we denote
\begin{equation}\label{eq:fat-realization} | \! | \! |  Z_\bullet| \! | \! | := \int^{[n]\in\Delta_{\mathrm{inj}}} R^n(\cst\tu_{\mV})\otimes Z_n \end{equation} 
its \emph{fat realization}. We write
\begin{equation}\label{eq:realization} |Z_\bullet| := R^{\bullet}(\cst\tu_{\mV}) \otimes_{\Delta} Z_{\bullet} \cong \int^{[n]\in\Delta} R^{n}(\cst\tu_{\mV})\otimes Z_n \end{equation} for its \emph{realization}.

\end{notation}

\begin{remark}
Suppose $\mV$ is a simplicial model category, then we may take $R^{n}(\cst\tu_{\mV}) = \Delta^n \otimes \tu_{\mV}$.
\end{remark}

\begin{lemma}\label{lem:fatreal}
  $| \! | \! |  Z_\bullet| \! | \! |  $ is cofibrant in $\mN$ when $Z_{\bullet}$ is objectwise cofibrant.
\end{lemma}
\begin{proof}
    The projective model structure on $\Fun(\Delta_{\mathrm{inj}},\mV^{})$ coincides with the Reedy one, therefore the result follows from \Cref{lem:leftquiliffobjcof}.
\end{proof}

\begin{proposition}\label{prop:explicit-projective-replacement-M}
Let $\mC$ be locally cofibrant. There is a $\mV$-functor $$Q \colon \mV\Fun(\mC,\mV) \to \mV\Fun(\mC,\mV)$$ $$QX:=| \! | \! |  B_\bullet^{\mC}X| \! | \! |$$ such that for every $X:\mC\to\mV$ objectwise cofibrant, $Q(X)$ is projectively cofibrant, and admits a natural objectwise weak equivalence $q_X:QX\to X$.
\end{proposition}

\begin{proof}
For a sequence $c_0,\ldots,c_n$, take $Y(c_0,\ldots,c_n) = \mC(c_{n-1},c_n)\otimes\cdots\otimes \mC(c_0,c_1)\otimes X(c_0).$ The functor $\mC(c_n,-)\otimes-:\mV\to \mV\Fun(\mC,\mV)_{\mathrm{proj}}$ is left Quillen and $Y(c_0,\ldots,c_n)$ is cofibrant (by our assumptions on $\mC$ and $X$). We obtain that the simplicial object in $\mV\Fun(\mC,\mV)_{\mathrm{proj}}$ given by $Z_{\bullet}=B_{\bullet}^{\mC}X(-)$ is degreewise projectively cofibrant. Thus, $| \! | \! |   B_\bullet^{\mC}X| \! | \! |  $ is projectively cofibrant by \Cref{lem:fatreal}. The extra degeneracy of $UB_\bullet^{\mC}X\to UX$ implies that $Uq_X \colon UQX\to UX$ is objectwise a weak equivalence, whence $q_X$ is too.
\end{proof}
\begin{remark}{\label{rem:locwellp}} A locally cofibrant $\mV$-category $\mC$ is \emph{locally well-pointed} if for every $c \in \mC$, the identity morphism $\id_{c}\ \colon \tu_{\mV}\to\mC(c,c)$ is a cofibration. In this case, in \ref{prop:explicit-projective-replacement-M} one may instead take $QX:=|B_\bullet^{\mC}X|$. If the unit of $\mV$ is terminal and every monomorphism is a cofibration, then every $\mC$ is locally well pointed.
\end{remark}

\end{document}